\documentclass[11pt]{article}
\usepackage[latin1]{inputenc}
\usepackage{amsmath,amsthm,amssymb}
\usepackage{amsfonts}
\usepackage{amsmath,amsthm,amssymb,amscd}
\usepackage{latexsym}
\usepackage{color}
\usepackage{graphicx}
\usepackage{mathrsfs}
\usepackage{cite}

\makeatletter \@addtoreset{equation}{section} \makeatother

\allowdisplaybreaks

\usepackage{color,enumitem,graphicx}
\usepackage[colorlinks=true,urlcolor=black,
citecolor=black,linkcolor=black,linktocpage,pdfpagelabels,
bookmarksnumbered,bookmarksopen]{hyperref}

\newtheorem{theorem}{Theorem}[section]
\newtheorem{definition}{Definition}[section]

\newtheorem{lemma}{Lemma}[section]
\newtheorem{remark}{Remark}[section]

\newtheorem{corollary}[theorem]{Corollary}

\allowdisplaybreaks

\begin{document}
\title{Normalized solutions to the biharmonic nonlinear  Schr\"{o}dinger equation with combined nonlinearities}

\author{Wenjing Chen\footnote{Corresponding author.}\ \footnote{E-mail address:\, {\tt wjchen@swu.edu.cn} (W. Chen), {\tt zxwangmath@163.com} (Z. Wang).}\  \ and Zexi Wang\\
\footnotesize  School of Mathematics and Statistics, Southwest University,
Chongqing, 400715, P.R. China}

\date{ }
\maketitle

\begin{abstract}
{In this article, we study the existence of normalized ground state solutions for the following  biharmonic nonlinear  Schr\"{o}dinger equation with combined nonlinearities
\begin{equation*}
  \Delta^2u=\lambda u+\mu|u|^{q-2}u+|u|^{p-2}u,\quad \text {in $\mathbb{R}^N$}
\end{equation*}
having prescribed mass
\begin{equation*}
  \int_{\mathbb{R}^N}|u|^2dx=a^2,
\end{equation*}
where $N\geq2$, $\mu\in \mathbb{R}$, $a>0$, $2<q<p<\infty$ if $2\leq N\leq 4$, $2<q<p\leq 4^*$ if $N\geq 5$, and $4^*=\frac{2N}{N-4}$ is the Sobolev critical exponent and $\lambda\in \mathbb{R}$ appears as a Lagrange multiplier. By using the Sobolev subcritical approximation method, we prove the second critical point of mountain pass type for the case $N\geq5$, $\mu>0$, $p=4^*$, and $2<q<2+\frac{8}{N}$. Moreover, we also consider the case $\mu=0$ and $\mu<0$.}



\end{abstract}

\section{Introduction and statement of main results}
In this paper, we are concerned with the following biharmonic nonlinear Schr\"{o}dinger equation with combined nonlinearities
\begin{align}\label{question}
  \begin{split}
  \left\{
  \begin{array}{ll}
    \Delta^2u=\lambda u+\mu|u|^{q-2}u+|u|^{p-2}u,\quad \text {in $\mathbb{R}^N$}, \\
    \displaystyle\int_{\mathbb{R}^N}|u|^2dx=a^2,\quad u\in H^2(\mathbb{R}^N),
    \end{array}
    \right.
  \end{split}
  \end{align}
where $N\geq2$, $\mu\in \mathbb{R}$, $a>0$, $2<q<p<\infty$ if $2\leq N\leq 4$, $2<q<p\leq 4^*$ if $N\geq 5$, and $4^*=\frac{2N}{N-4}$ is the Sobolev critical exponent and $\lambda\in \mathbb{R}$ appears as a Lagrange multiplier.

The interest in studying \eqref{question} comes from seeking the standing waves with the form $\psi(x,t)=e^{-i \lambda t}u(x)$ of the following time-dependent nonlinear biharmonic Schr\"{o}dinger equation
\begin{equation*}\label{stand}
  i\partial _t\psi -\Delta^2\psi+\mu|\psi|^{q-2}\psi+|\psi|^{p-2}\psi=0, \quad \psi(0,x)=\psi_0\in H^2(\mathbb{R}^N),\quad (t,x)\in \mathbb{R}\times \mathbb{R}^N.
\end{equation*}
This equation was considered in \cite{IK,Tur} to study the stability of solitons in magnetic materials once the effective quasi particle mass becomes infinite. Moreover, the biharmonic operator $\Delta^2$ was also used in \cite{K,KS} to reveal the effects of high-order terms in the nonlinear Schr\"{o}dinger equations with the mixed dispersion.  

The type of solution for \eqref{question} is called normalized solution, and in this paper, we focus on this issue. In particular, we are interested in
looking for ground state solutions, i.e., solutions minimizing the associated energy functional among
all nontrivial solutions, and the associated energy is called ground state energy.

In recent years, much attention has been paid to study the normalized solutions for the following nonlinear Schr\"{o}dinger equation, see  \cite{AJM,BM,DZ,Li1,Jeanjean,JJLV,JL0,JL,Soave1,Soave2,WW},
\begin{align}\label{lap}
  \begin{split}
  \left\{
  \begin{array}{ll}
   -\Delta u=\lambda u+f(u),\quad \text {in $\mathbb{R}^N$} , \\
    \displaystyle\int_{\mathbb{R}^N}|u|^2dx=a^2,\quad u\in H^1(\mathbb{R}^N).
    \end{array}
    \right.
  \end{split}
  \end{align}
In particular, Soave \cite{Soave1} considered the ground state solutions of \eqref{lap} with combined nonlinearities $f(u)=\mu|u|^{q-2}u+|u|^{p-2}u$, $2<q\leq 2+\frac{4}{N}\leq p<2^*$ and $q<p$, where $2^*=\infty$ if $N\leq2$ and $2^*=\frac{2N}{N-2}$ if $N\geq3$. The Sobolev critical case $p=2^*$ was considered by Soave in \cite{Soave2} when $N\geq3$, and some open questions were proposed, one of them is:

($Q'$) Does \eqref{lap} has a critical point of mountain pass type in the case $\mu>0$ and  $2<q<2+\frac{4}{N}$?\\
Jeanjean and Le \cite{JL0} gave a positive answer to the question $(Q')$ for $N\geq4$, Wei and Wu \cite{WW} gave a positive answer to $(Q')$ for $N=3$.  Later, Luo and Zhang \cite{LZ}, Zhang and Han \cite{ZH}, Zhen and Zhang \cite{ZZ} generalized the above results to the fractional laplace equations with combined nonlinearities.

For some related studies on the biharmonic equations with the mixed dispersion, we refer the readers to \cite{BCdN,BCGJ,BFJ,FJMM,LY2,LZ1,LZZ} and references therein.  As far as we know, there are only a few papers to study \eqref{question} besides \cite{CHMS,MC}. In \cite{MC}, by analyzing the behavior of the ground
state energy with respect to the prescribed mass, Ma and Chang obtained a ground state solution $\hat{u}$, which is an interior local minimizer of the functional \begin{equation*}
  \mathcal{J}_{4^*,q}(u)=\frac{1}{2}\int_{\mathbb{R}^N}|\Delta u|^2dx-\frac{1}{4^*}\int_{\mathbb{R}^N}|u|^{4^*}dx-\frac{\mu}{q}\int_{\mathbb{R}^N}|u|^qdx
\end{equation*}
in the set
\begin{equation*}
  \mathcal{A}_r:=\Big\{u\in S_a:\|\Delta u\|_2^2<r\}
\end{equation*}
for a suitable $r>0$, when $N\geq5$, $\mu>0$, $2<q<2+\frac{8}{N}$ and
\begin{equation}\label{con3}
\mu a^{q-\frac{N(q-2)}{4}}<\Big(\frac{1}{A+B}\Big)^{\frac{4Nq-N^2(q-2)}{32}},
\end{equation}
where
\begin{equation*}
 S_a:=\Big\{u\in H^2(\mathbb{R}^N):\int_{\mathbb{R}^N}|u|^2dx=a^2\Big\},
\end{equation*}
\begin{equation*}
  A:=\frac{2\mu B_{N,q}}{q}\Big[\frac{\mu (N-4)(8+2N-Nq) B_{N,q}4^*}{32qS^{4^*}}\Big]^{\frac{(N-4)(Nq-2N-8)}{N[4q-N(q-2)]}},
\end{equation*}
\begin{equation*}
 B:=\frac{2\cdot4^*}{4^*}\Big[\frac{\mu (N-4)(8+2N-Nq) B_{N,q}4^*}{32qS^{4^*}}\Big]^{\frac{32}{N[4q-N(q-2)]}},
\end{equation*}
$B_{N,q}$, $S$ are the optimal constants to the Gagliardo-Nirenberg inequality \cite{Nirenberg1} and Sobolev inequality respectively \cite{NSY}. Furthermore, all ground state solutions are proved to be local minima of $\mathcal{J}_{4^*,q}$.
Therefore, it is natural for us to ask:

($Q$) Does \eqref{question} has a second critical point of mountain pass type for $\mathcal{J}_{4^*,q}|_{S_a}$ in the case $N\geq5$, $\mu>0$, $p=4^*$, and $2<q<2+\frac{8}{N}$?
\\Inspired by \cite{JL0} and \cite{WW}, by using the Sobolev subcritical approximation method developed in \cite{Li1}, which is used to study the existence of ground state solutions for \eqref{lap} when $f(u)=\mu|u|^{q-2}u+|u|^{p-2}u$, $N\geq3$, $\mu>0$, $2+\frac{4}{N}<q<2^*$ and $p=2^*$, we give a positive answer to $(Q)$ in Theorem \ref{th8}.

In \cite{CHMS}, by adapting the arguments used in \cite{BS} and analyzing the behavior of the ground
state energy with respect to the prescribed mass, especially the monotonicity of the function $a\mapsto M_{4^*,q,a}$ (defined in \eqref{infP}),  Chang et al. considered \eqref{question} when $N\geq5$, $\mu>0$, $2+\frac{8}{N}\leq q< p\leq4^*$, and obtained the existence of ground state solutions. The Sobolev critical case $p=4^*$ was also completed by constructing an appropriate testing function and a precise estimation. Moreover, they proved that the standing waves associated with the ground state solutions to \eqref{question} are strongly unstable in $H^2_{rad}(\mathbb{R}^N)$ by blowup when $2+\frac{8}{N}\leq q< p\leq4^*$, where $H^2_{rad}(\mathbb{R}^N)$ denotes the subset of the radially symmetric functions in $H^2(\mathbb{R}^N)$. 

Motivated by the results mentioned above, considering that there is no contribution devoted to study \eqref{question} with $\mu=0$ or $\mu<0$, the rest of this paper is aimed to consider the case $\mu=0$ and $\mu<0$.

Solutions of  \eqref{question} can be obtained by finding critical points of the energy functional $ \mathcal{J}_{p,q}:H^2(\mathbb{R}^N)\rightarrow \mathbb{R}$ on $S_a$.
\begin{equation*}
  \mathcal{J}_{p,q}(u)=\frac{1}{2}\int_{\mathbb{R}^N}|\Delta u|^2dx-\frac{1}{p}\int_{\mathbb{R}^N}|u|^pdx-\frac{\mu}{q}\int_{\mathbb{R}^N}|u|^qdx,
\end{equation*}
where $H^2(\mathbb{R}^N)$
is the Sobolev space defined by
\begin{equation*}
  H^2(\mathbb{R}^N)=\Big\{u\in L^2(\mathbb{R}^N):\nabla u\in L^2(\mathbb{R}^N),\,\Delta u\in L^2(\mathbb{R}^N)\Big\},
\end{equation*}
with the norm
\begin{equation*}
  \|u\|=\Big(\int_{\mathbb{R}^N}(|u|^2+2|\nabla u|^2+|\Delta u|^2)dx\Big)^{\frac{1}{2}}.
\end{equation*}

Define
\begin{equation*}
  m_{p,q,a}:=\inf\limits_{u\in S_a}\mathcal{J}_{p,q}(u),
\end{equation*}
and
\begin{equation*}
\bar{p}:=2+\frac{8}{N}.
\end{equation*}

To state our main results, we need
the following lemmas.
\begin{lemma}\cite{BFV}
For $N\geq1$ and $2<p<\frac{2N}{(N-4)^+}$, the following biharmonic equation exists a ground state solution $Q_{N,p}$, while the uniqueness for $Q_{N,p}$ still open for $N\geq2$, except in one dimension \cite{FL}, 
\begin{equation*}\label{gssolu}
  \frac{N(p-2)}{8}\Delta ^2 Q+\Big[1+\frac{(4-N)(p-2)}{8}\Big] Q-|Q|^{p-2}Q=0, \quad Q\in H^2(\mathbb{R}^N).
\end{equation*}
\end{lemma}
When $p=\bar{p}$, we define $\bar{a}:=\|Q_{N,\bar{p}}\|_2$.

\begin{lemma}(Gagliardo-Nirenberg inequality) \cite{Nirenberg1}\label{GN}
For $N\geq1$, $2<p<\frac{2N}{(N-4)^+}$ and $u\in H^2(\mathbb{R}^N)$, there exists an optimal constant $B_{N,p}>0$ depending on $N,p$ such that
\begin{equation*}
  \|u\|_p^p\leq B_{N,p}\| u\|_2^{p(1-\gamma_p)}\|\Delta u\|_2^{p\gamma_p},
\end{equation*}
where \begin{equation*}
\gamma_p:=\frac{N(p-2)}{4p}.
\end{equation*}
Moreover, $B_{N,p}$ can be achieved by $Q_{N,p}$ and $B_{N,p}=\frac{p}{2\|Q_{N,p}\|_2^{p-2}}$. 
\end{lemma}

\begin{lemma}(Sobolev inequality)\label{Sob} \cite{NSY}
For $N\geq5$ and $u\in D^{2,2}(\mathbb{R}^N)$, there exists an optimal constant $S>0$ depending on $N$ such that
\begin{equation*}
  S\|u\|_{4^*}^2\leq \|\Delta u\|_2^2,
\end{equation*}
where $D^{2,2}(\mathbb{R}^N)$ denotes the completion of $C^\infty_0(\mathbb{R}^N)$ with respect to the norm $\|u\|_{D^{2,2}}:=\|\Delta u\|_2$. Moreover, $S$ can be achieved by any constant multiple of the function
\begin{equation}\label{deU}
  U_{\varepsilon,y}=[N(N-4)(N^2-4)]^{\frac{N-4}{8}}\Big(\frac{\varepsilon}{\varepsilon^2+|x-y|^2}\Big)^{\frac{N-4}{2}}, \quad \varepsilon>0,\,y\in \mathbb{R}^N,
\end{equation}
which are solutions to
\begin{equation}\label{criequ}
  \Delta^2 u=u^{4^*-1}, \quad \text{in $\mathbb{R}^N$}.
\end{equation}
\end{lemma}
In the first part, we consider \eqref{question} with $\mu=0$, and denote the associated energy functional by $\mathcal{J}_{p}$ and $m_{p,a}=\inf\limits_{u\in S_a}\mathcal{J}_{p}(u)$.

\begin{theorem}\label{th1}
Let $N\geq2$, $2<p<\frac{2N}{(N-4)^+}$, $\mu=0$,

(i) If $2<p<\bar{p}$ and $p\in 2\mathbb{N}$, then $-\infty<m_{p,a}<0$ and $m_{p,a}$ has a radial minimizer $\hat{u}$. In particular, $\hat{u}$ is a ground state solution of \eqref{question} with some $\hat{\lambda}<0$. 

(ii) If $p=\bar{p}$,

\quad \,(a) for $0<a<\bar{a}$, then $m_{p,a}=0$ and \eqref{question} has no solution. In particular, $m_{p,a}$ cannot be achieved by any $u\in S_a$, namely,  \eqref{question} has no ground state solution.

\quad \,(b) for $a=\bar{a}$, then $m_{p,a}=0$ and $m_{p,a}$ has a minimizer $Q_{N,\bar{p}}$. In particular, $Q_{N,\bar{p}}$ is a ground state solution of \eqref{question} with some $\bar{\lambda}<0$.

\quad \,(c) for $a>\bar{a}$, $m_{p,a}=-\infty$.

(iii) If $\bar{p}<p<\frac{2N}{(N-4)^+}$, then $m_{p,a}=-\infty$. 
\end{theorem}

\begin{theorem}\label{th7}
Let $N\geq9$, $p=4^*$, $\mu=0$, then \eqref{question} has a unique positive ground state solution $U_{\varepsilon,0}$ at level $\frac{2}{N}S^{\frac{N}{4}}$, where $U_{\varepsilon,0}$ is defined in \eqref{deU} for the unique choice of $\varepsilon>0$ such that $\|U_{\varepsilon,0}\|_2=a$. The function $U_{\varepsilon,0}$ also solves \eqref{question} for $\lambda=0$.  
\end{theorem}

In the second part, we consider \eqref{question} with $\mu\neq0$, and we have:

\begin{theorem}\label{th2}
Let $N\geq2$, $2<q<p<\bar{p}$, $\mu>0$, and $p,q\in 2\mathbb{N}$, for $a>0$, then $-\infty<m_{p,q,a}<0$ and $m_{p,q,a}$ has a radial minimizer $\hat{u}$. In particular, $\hat{u}$ is a ground state solution of \eqref{question} with some $\hat{\lambda}<0$. 
\end{theorem}

\begin{theorem}\label{th3}
Let $N\geq2$, $2<q<p=\bar{p}$,

(i) If $0<a<\bar{a}$,

\quad\, (a) for $\mu>0$ and $p,q\in 2\mathbb{N}$, then $-\infty<m_{p,q,a}<0$ and $m_{p,q,a}$ has a radial minimizer $\hat{u}$. In particular, $\hat{u}$ is a ground state solution of \eqref{question} with some $\hat{\lambda}<0$.

\quad\, (b) for $\mu<0$, then $m_{p,q,a}=0$ and \eqref{question} has no solution.

(ii) If $a=\bar{a}$,

\quad\, (a) for $\mu>0$, $m_{p,q,a}=-\infty$. 

\quad\, (b) for $\mu<0$, then $m_{p,q,a}=0$ and \eqref{question} has no solution. 

(iii) If $a>\bar{a}$, for any $\mu\neq0$, $m_{p,q,a}=-\infty$.
\end{theorem}

\begin{theorem}\label{th4}
Let $N\geq2$, $2<q<\bar{p}<p<\frac{2N}{(N-4)^+}$, $\mu>0$ and $p,q\in 2\mathbb{N}$, if
\begin{equation}\label{con1}
  \mu a^{\gamma(p,q)}<\Big(\frac{p(\bar{p}-q)}{2B_{N,p}(p-q)}\Big)^{\frac{\bar{p}-q}{p-\bar{p}}}\Big(\frac{q(p-\bar{p})}{2B_{N,q}(p-q)}\Big)
\end{equation}
with
\begin{equation}\label{gamma}
  \gamma(p,q)=\frac{\bar{p}-q}{p-\bar{p}}\Big(p-\frac{N(p-2)}{4}\Big)+\Big(q-\frac{N(q-2)}{4}\Big)>0.
\end{equation}
Then \eqref{question} has two radial solutions $\hat{u}$ and $\tilde{u}$ for suitable $\hat{\lambda},\tilde{\lambda}<0$, where $\hat{u}$ is an interior local minimizer of $\mathcal{J}_{p,q}$  in the set $\mathcal{A}_k$
for a suitable $k>0$ at level $c_{p,q,a}<0$ (defined in \eqref{jubu}), $\tilde{u}$ is a critical point of mountain pass type for $\mathcal{J}_{p,q}|_{S_{a,r}}$ at level $\sigma_{p,q,a,r}>0$ (defined in \eqref{shanlu1}), $S_{a,r}:=S_a\cap H^2_{rad}(\mathbb{R}^N)$. In particular, $\hat{u}$ is a ground state solution of \eqref{question}.
\end{theorem}

\begin{theorem}\label{th5}
Let $N\geq2$, $2<q\leq \bar{p}<p<\frac{2N}{(N-4)^+}$, $\mu<0$, if
\begin{equation}\label{con2}
  |\mu| a^{\gamma(p,q)}<\Big(\frac{4p}{NB_{N,p}(p-2)}\Big)^{\frac{\bar{p}-q}{p-\bar{p}}}\Big(\frac{q(4p+2N-Np)}{2NB_{N,q}(p-q)}\Big),
\end{equation}
where $\gamma(p,q)$ is defined in \eqref{gamma}. Then \eqref{question} has a radial ground state solution at positive level for some $\lambda<0$.
\end{theorem}

\begin{theorem}\label{th8}
Let $N\geq7$, $2<q<\bar{p}<p=4^*$, $\mu>0$, and \eqref{con3}, \eqref{con1} hold. If $N=7$, we further assume that $q>\frac{8}{3}$.
Then \eqref{question} has a second solution $\tilde{u}$ at level $\sigma_{4^*,q,a,r}>0$ (defined in \eqref{shanlu2}) for some $\tilde{\lambda}<0$. Moreover, $\tilde{u}$ is radial and is a critical point of mountain pass type for $\mathcal{J}_{4^*,q}|_{S_{a,r}}$.
\end{theorem}

\begin{remark}
{\rm Follow the arguments of \cite{CHMS}, it's easy to find that, in the case of $(iii)$ in Theorem \ref{th1}, \eqref{question} has a ground state solution $\tilde{u}$ at positive level for some $\tilde{\lambda}<0$.}
\end{remark}

\begin{remark}
{\rm The crucial point of Theorem \ref{th8} is to obtain a good energy estimate of $\sigma_{4^*,q,a,r}$ such that the compactness of $\{u_{p_n}^-\}$ (defined in \eqref{fu}) holds in $D^{2,2}(\mathbb{R}^N)$, and the threshold of such compactness should be $M_{4^*,q,a,r}+\frac{2}{N}S^{\frac{N}{4}}$, where $p_n\in(2,4^*)$, $\lim\limits_{n\rightarrow\infty}p_n=4^*$, and $M_{4^*,q,a,r}$ is defined in \eqref{M4^*}. }
\end{remark}

\section{Preliminaries}\label{sec preliminaries}
In this section, we give some preliminaries.

\subsection{Several important lemmas}\label{implemma}
\begin{lemma}\label{Four}\cite{BL}
For $N\geq1$ and $u\in L^2(\mathbb{R}^N)$, let $u^\sharp$ be the Fourier rearrangement of $u$ defined by
\begin{equation*}
  u^\sharp:=\mathcal{F}^{-1}((\mathcal{F}u)^*),
\end{equation*}
where $\mathcal{F}$ and $\mathcal{F}^{-1}$ denote the Fourier transform and the inverse Fourier transform respectively, and $f^*$ stands for the Schwarz rearrangement of $f$. Then $u^\sharp$ is radial, $\|u^\sharp\|_2=\|u\|_2$,  and $\|\Delta u^\sharp\|_2 \leq \|\Delta u\|_2$. Moreover, if $p\in 2\mathbb{N}$, we have
\begin{equation*}
 \|u^\sharp\|_{p}  \geq \|u\|_{p}.
\end{equation*}
\end{lemma}

\begin{lemma}\label{limits}
Let $N\geq5$ and $2<p<4^*$, then $\lim\limits_{p\rightarrow 4^*}B_{N,p}^{\frac{1}{p}}=S^{-\frac{1}{2}}$.
\end{lemma}
\begin{proof}
For $u\in H^2(\mathbb{R}^N)$, by Lemma \ref{GN}, we have
\begin{equation*}
  \|u\|_p\leq B_{N,p}^{\frac{1}{p}}\| u\|_2^{1-\gamma_p}\|\Delta u\|_2^{\gamma_p}.
\end{equation*}
For $2<p<4^*$, since
\begin{equation*}
  |u|^p\leq |u|^2+|u|^{4^*},
\end{equation*}
using the Lebesgue dominated convergence theorem, we obtain $\lim\limits_{p\rightarrow 4^*}\|u\|_p=\|u\|_{4^*}$. Noted that $\gamma_p\rightarrow 1$ as $p\rightarrow 4^*$, thus
\begin{equation*}
  \|u\|_{4^*}\leq \liminf\limits_{p\rightarrow 4^*}B_{N,p}^{\frac{1}{p}}\|\Delta u\|_2,
\end{equation*}
which implies that $S^{-\frac{1}{2}}\leq \liminf\limits_{p\rightarrow 4^*}B_{N,p}^{\frac{1}{p}}$. On the other hand, for $u\in H^2(\mathbb{R}^N)$, by Lemma \ref{Sob} and H\"{o}lder inequality, we have
\begin{equation*}
  \|u\|_p\leq \|u\|_2^{1-\gamma_p}\|u\|_{4^*}^{\gamma_p}\leq S^{-\frac{\gamma_p}{2}}\|u\|_2^{1-\gamma_p}\|\Delta u\|_2^{\gamma_p},
\end{equation*}
so $S^{-\frac{\gamma_p}{2}}\geq B_{N,p}^{\frac{1}{p}}$, which means $S^{-\frac{1}{2}}\geq \limsup\limits_{p\rightarrow 4^*}B_{N,p}^{\frac{1}{p}}$.
\end{proof}

\begin{lemma}\label{biancon}
Let $\{u_n\}\subset H^2(\mathbb{R}^N)$, $\{\tau_n\}\subset\mathbb{R}$, assume that $u_n\rightarrow u$ in $H^2(\mathbb{R}^N)$, and $\tau_n\rightarrow \tau$ in $\mathbb{R}$. Then $\tau_n*u_n\rightarrow \tau*u$ in $H^2(\mathbb{R}^N)$ as $n\rightarrow\infty$.
\end{lemma}
\begin{proof}
First, we prove that $\tau_n*u_n\rightharpoonup \tau*u$ in $H^2(\mathbb{R}^N)$. Taking any $\varphi\in C_0^\infty(\mathbb{R}^N)$, and for a compact set $K$ containing the support of $\varphi(e^{-\tau_n}\cdot)$ for any $n\in \mathbb{N}$ large enough, using the Lebesgue dominated convergence theorem, we get
\begin{align*}
  \int_{\mathbb{R}^N}e^{\frac{N}{2}\tau_n}u_n(e^{\tau_n}x)\varphi(x)dx&= \int_{K}e^{-\frac{N}{2}\tau_n}u_n(y)\varphi(e^{-\tau_n}y)dy\\
  &\rightarrow \int_{K}e^{-\frac{N}{2}\tau}u(y)\varphi(e^{-\tau}y)dy=\int_{\mathbb{R}^N}e^{\frac{N}{2}\tau}u(e^{\tau}x)\varphi(x)dx,\quad\text{as $n\rightarrow\infty$}.
\end{align*}
Similarly, we can show that for any $i=1,2,\ldots  , N$ and $\varphi\in C_0^\infty(\mathbb{R}^N)$
\begin{align*}
  \int_{\mathbb{R}^N}\varphi \frac{\partial (e^{\frac{N}{2}\tau_n}u_n(e^{\tau_n}x))}{\partial {x_i}}dx\rightarrow \int_{\mathbb{R}^N}\varphi \frac{\partial (e^{\frac{N}{2}\tau}u(e^{\tau}x))}{\partial {x_i}}dx,\quad\text{as $n\rightarrow\infty$},
\end{align*}
and
\begin{align*}
  \int_{\mathbb{R}^N}\varphi \frac{\partial ^2 (e^{\frac{N}{2}\tau_n}u_n(e^{\tau_n}x))}{\partial {x_i}^2}dx\rightarrow \int_{\mathbb{R}^N}\varphi \frac{\partial ^2(e^{\frac{N}{2}\tau}u(e^{\tau}x))}{\partial {x_i}^2}dx,\quad\text{as $n\rightarrow\infty$}.
\end{align*}
As a consequence, $\tau_n*u_n\rightharpoonup \tau*u$ in $H^2(\mathbb{R}^N)$. Furthermore,
\begin{align*}
  \|e^{\frac{N}{2}\tau_n}u_n(e^{\tau_n}x)\|^2&= \int_{\mathbb{R}^N}|u_n|^2dx+2e^{2\tau_n}\int_{\mathbb{R}^N}|\nabla u_n|^2dx+e^{4\tau_n}\int_{\mathbb{R}^N}|\Delta u_n|^2dx\\
  &\rightarrow \int_{\mathbb{R}^N}|u|^2dx+2e^{2s}\int_{\mathbb{R}^N}|\nabla u|^2dx+e^{4s}\int_{\mathbb{R}^N}|\Delta u|^2dx=\|e^{\frac{N}{2}\tau}u_(e^{\tau}x)\|^2,\quad\text{as $n\rightarrow\infty$},
\end{align*}
thus $\|\tau_n*u_n\|\rightarrow \|\tau*u\|$ as $n\rightarrow\infty$, and we conclude the result.
\end{proof}

\begin{lemma}\cite[Lemma 4.8]{Kav}\label{weakcon}
Let $\Omega\subseteq\mathbb R^{4}$ be any open set. For $1<s<\infty$, if $\{u_n\}$ be bounded in $L^s(\Omega)$ and $u_n(x)\rightarrow u(x)$ a.e. in $\Omega$, then $u_n(x)\rightharpoonup u(x)$ in $L^s(\Omega)$.
\end{lemma}
\subsection{Decomposition of Poho\u{z}aev manifold}\label{pohoz}
We have the following Poho\u{z}aev identity.
\begin{lemma}\label{poho}\cite[Lemma 2.1]{BCGJ1}
Let $N\geq1$ and $2<p,q<\frac{2N}{(N-4)^+}$, if $u\in H^2(\mathbb{R}^N)$ is a weak solution of
\begin{equation*}
  \Delta^2u=\lambda u+\mu|u|^{q-2}u+|u|^{p-2}u,
\end{equation*}
then $u$ satisfies the Poho\u{z}aev identity
\begin{equation}\label{pohoim}
  \frac{N-4}{2}\|\Delta u\|_2^2-\frac{\lambda N}{2}\|u\|_2^2-\frac{\mu N}{q}\|u\|_q^{q}-\frac{N}{p}\|u\|_p^{p}=0.
\end{equation}
As a consequence, it holds that
\begin{equation*}
  \|\Delta u\|_2^2-\mu\gamma_q\|u\|_q^{q}-\gamma_p\|u\|_p^{p}=0.
\end{equation*}

\end{lemma}
For any $u\in S_a$ and $\tau\in \mathbb{R}$, we define
\begin{equation*}
  (\tau*u)(x):=e^{\frac{N}{2}\tau}u(e^\tau x),\quad \text{for a.e. $x\in \mathbb{R}^N$},
\end{equation*}
and the fiber maps
\begin{equation*}
  \Psi_u(\tau):= \mathcal{J}_{p,q}(\tau*u)=\frac{e^{4\tau}}{2}\int_{\mathbb{R}^N}|\Delta u|^2dx-\frac{e^{\frac{N(p-2)}{2}\tau}}{p}\int_{\mathbb{R}^N}| u|^pdx-\frac{\mu e^{\frac{N(q-2)}{2}\tau}}{q}\int_{\mathbb{R}^N}| u|^qdx.
\end{equation*}
Then $\tau*u \in S_a$ and $ \Psi_u(\tau)$ is coercive on $S_a$ for $p,q<\bar{p}$, while $ \Psi_u(\tau)$ is not bounded from below on $S_a$ for $p>\bar{p}$ or $q>\bar{p}$, i.e., $m_{p,q,a}=-\infty$. Based on this fact, we call $\bar{p}$ the $L^2$-critical exponent.
In the case of $p>\bar{p}$ or $q>\bar{p}$, a special role will played by the Poho\u{z}aev manifold
\begin{equation*}
\mathcal{P}_{p,q,a}:=\Big\{u\in S_a:P_{p,q}(u)=0\Big\},
\end{equation*}
where
\begin{align*}
  P_{p,q}(u)=\int_{\mathbb{R}^4}|\Delta u |^2dx-\gamma_p \int_{\mathbb{R}^4}|u |^pdx-\mu\gamma_q \int_{\mathbb{R}^4}|u |^qdx.
\end{align*}
From Lemma \ref{poho}, we find that any critical point of $\mathcal{J}_{p,q}|_{S_a}$ stays in $\mathcal{P}_{p,q,a}$.  Moreover, if $\mathcal{P}_{p,q,a}^0=\emptyset$, then $\mathcal{P}_{p,q,a}$ is a natural constraint, which means that if $u\in \mathcal{P}_{p,q,a}$ is a critical point for $\mathcal{J}_{p,q}|_{\mathcal{P}_{p,q,a}}$, then $u$ is a critical point for $\mathcal{J}_{p,q}|_{S_a}$ (see Lemma \ref{natural}). This  together with the fact that $m_{p,q,a}=-\infty$ enlightens us to consider the minimization of $\mathcal{J}_{p,q}$ on ${\mathcal{P}_{p,q,a}}$
\begin{equation}\label{infP}
  M_{p,q,a}:=\inf\limits_{u\in \mathcal{P}_{p,q,a}}\mathcal{J}_{p,q}(u).
\end{equation}
We shall see  that the critical points of $\Psi_u $ allow to project a function on $\mathcal{P}_{p,q,a}$. Thus monotonicity and convexity properties of $\Psi_u $ strongly affect the structure of $\mathcal{P}_{p,q,a}$, and in turn the geometry of $\mathcal{J}_{p,q}|_{S_a}$. In this direction, we consider the decomposition of $\mathcal{P}_{p,q,a}$ into the disjoint union $\mathcal{P}_{p,q,a}=\mathcal{P}_{p,q,a}^+\cup \mathcal{P}_{p,q,a}^0\cup\mathcal{P}_{p,q,a}^-$, where
\begin{align*}
  &\mathcal{P}_{p,q,a}^+:=\Big\{u\in \mathcal{P}_{p,q,a}:2\|\Delta u\|_2^2>p\gamma_p^2\|u\|_p^p+\mu q\gamma_q^2\|u\|_q^q\Big\}=\Big\{u\in \mathcal{P}_{p,q,a}:\Psi_u''(0)>0\Big\},\\
  &\mathcal{P}_{p,q,a}^0:=\Big\{u\in \mathcal{P}_{p,q,a}:2\|\Delta u\|_2^2=p\gamma_p^2\|u\|_p^p+\mu q\gamma_q^2\|u\|_q^q\Big\}=\Big\{u\in \mathcal{P}_{p,q,a}:\Psi_u''(0)=0\Big\},\\
  &\mathcal{P}_{p,q,a}^-:=\Big\{u\in \mathcal{P}_{p,q,a}:2\|\Delta u\|_2^2<p\gamma_p^2\|u\|_p^p+\mu q\gamma_q^2\|u\|_q^q\Big\}=\Big\{u\in \mathcal{P}_{p,q,a}:\Psi_u''(0)<0\Big\}.
\end{align*}

For $\mu=0$, we denote the associated fiber maps by $\Psi_{u,0}$, $\mathcal{P}_{p,a}=\Big\{u\in S_a:P_{p}(u)=0\Big\}$, and $M_{p,a}=\inf\limits_{u\in \mathcal{P}_{p,a}}\mathcal{J}_{p}(u)$, where
\begin{align*}
  P_{p}(u)=\int_{\mathbb{R}^4}|\Delta u |^2dx-\gamma_p \int_{\mathbb{R}^4}|u |^pdx.
\end{align*}
The decomposition of $\mathcal{P}_{p,a}$ into the disjoint union $\mathcal{P}_{p,a}=\mathcal{P}_{p,a}^+\cup \mathcal{P}_{p,a}^0\cup\mathcal{P}_{p,a}^-$, where
\begin{align*}
  &\mathcal{P}_{p,a}^+:=\Big\{u\in \mathcal{P}_{p,a}:2\|\Delta u\|_2^2>p\gamma_p^2\|u\|_p^p\Big\}=\Big\{u\in \mathcal{P}_{p,a}:\Psi_{u,0}''(0)>0\Big\},\\
  &\mathcal{P}_{p,a}^0:=\Big\{u\in \mathcal{P}_{p,a}:2\|\Delta u\|_2^2=p\gamma_p^2\|u\|_p^p\Big\}=\Big\{u\in \mathcal{P}_{p,a}:\Psi_{u,0}''(0)=0\Big\},\\
  &\mathcal{P}_{p,a}^-:=\Big\{u\in \mathcal{P}_{p,a}:2\|\Delta u\|_2^2<p\gamma_p^2\|u\|_p^p\Big\}=\Big\{u\in \mathcal{P}_{p,a}:\Psi_{u,0}''(0)<0\Big\}.
\end{align*}
\begin{lemma}\label{onP}
Let $u\in S_a$, then $\tau\in \mathbb{R}$ is a critical point of  $\Psi_u$ if and only if $\tau*u \in \mathcal{P}_{p,q,a}$. Similarly, $\tau\in \mathbb{R}$ is a critical point of  $\Psi_{u,0}$ if and only if $\tau*u \in \mathcal{P}_{p,a}$.
\end{lemma}
\begin{proof}
The proof is standard, we omit it.
\end{proof}
\begin{lemma}\label{natural}
If $\mathcal{P}_{p,q,a}^0=\emptyset$, then $\mathcal{P}_{p,q,a}$ is a natural constraint. Similarly, if $\mathcal{P}_{p,a}^0=\emptyset$, then $\mathcal{P}_{p,a}$ is a natural constraint. 
\end{lemma}
\begin{proof}
If $u\in \mathcal{P}_{p,q,a}$ is a critical point for $\mathcal{J}_{p,q}|_{\mathcal{P}_{p,q,a}}$, then by the Lagrange multiplies rule \cite[Proposition 5.12]{Wil}, there exists $\lambda,\nu\in \mathbb{R}$ such that
\begin{equation*}
  \mathcal{J}_{p,q}'=\lambda u+\nu P_{p,q}',\quad x\in \mathbb{R}^N,
\end{equation*}
i.e.,
\begin{equation*}
  \Delta^2u=|u|^{p-2}u+\mu|u|^{q-2}u+\lambda u+\nu (\Delta^2u-\frac{p\gamma_p}{2}|u|^{p-2}u-\frac{\mu q\gamma_q}{2}|u|^{q-2}u),\quad x\in \mathbb{R}^N,
\end{equation*}
that is
\begin{equation*}
  (1-\nu)\Delta^2u=\lambda u+(1-\frac{\nu p\gamma_p}{2})|u|^{p-2}u+\mu(1-\frac{\nu q\gamma_q}{2})|u|^{q-2}u,\quad x\in \mathbb{R}^N.
\end{equation*}
We have to prove that  $\nu=0$. By Lemma \ref{poho}, we have
\begin{equation*}
   (1-\nu)\| \Delta u\|_2^2=(1-\frac{\nu p\gamma_p}{2})\gamma_p\|u\|_p^p+\mu(1-\frac{\nu q\gamma_q}{2})\gamma_q\|u\|_q^q.
\end{equation*}
Since $u\in \mathcal{P}_{p,q,a}$, we have
\begin{equation*}
  \| \Delta u\|_2^2=\gamma_p\|u\|_p^p+\mu\gamma_q\|u\|_q^q.
\end{equation*}
From the above equalities, we deduce that
\begin{equation*}
  \nu(2 \| \Delta u\|_2^2-p\gamma_p^2\|u\|_p^p-\mu q\gamma_q^2\|u\|_q^q)=0.
\end{equation*}
But the term inside the bracket cannot be $0$, since $\mathcal{P}_{p,q,a}^0=\emptyset$, so $\nu=0$. 
\end{proof}

\subsection{A general minimax theorem}
In this subsection, we shall give a general minimax theorem to establish the existence of a Palais-Smale sequence.
\begin{definition}\cite[Definition 5.1]{G}
Let $B$ be a closed subset of $X$. A class of $\mathcal{F}$ of compact subsets of $X$ is a homotopy stable family with extended boundary $B$
if for any set $A\in \mathcal{F}$ and any $\eta\in C([0,1]\times X,X)$ satisfying $\eta(t,x)=x$ for all $(t,x)\in (\{0\}\times X)\cup([0,1]\times B)$, one has $\eta(\{1\}\times A)\in\mathcal{F}$.
\end{definition}
\begin{lemma}\cite[Theorem 5.2]{G}\label{Ghouss}
Let $\varphi$ be a $C^1$-functional on a complete connected $C^{1}$-Finsler manifold $X$, and consider a homotopy stable family $\mathcal{F}$ with an extended closed boundary $B$. Set $c=c(\varphi,\mathcal{F})=\inf\limits_{A\in \mathcal{F}}\max\limits_{x\in A}\varphi(x)$ and let $\mathcal{F}$ be a closed subset of $X$ satisfying
\begin{equation*}
  A\cap F\backslash B\neq\emptyset, \quad\text{for any $A\in \mathcal{F}$}
\end{equation*}
and
\begin{equation*}
  \sup\limits_{x\in B}\varphi(x)\leq c\leq \inf\limits_{x\in F}\varphi(x) .
\end{equation*}
Then, for any sequence of sets $\{A_n\} \subset\mathcal{F}$ satisfying $\lim\limits_{n\rightarrow\infty}\sup\limits_{x\in A_n}\varphi(x)=c$, there exists a sequence $\{x_n\}\subset X\backslash B$ such that
\\$(i)$ $\lim\limits_{n\rightarrow\infty}\varphi(x_n)=c$; 
\\$(ii)$ $\lim\limits_{n\rightarrow\infty}\|\varphi'(x_n)\|_*=0$;
\\$(iii)$ $\lim\limits_{n\rightarrow\infty}dist(x_n,\mathcal{F})=0$; 
\\$(iii)$ $\lim\limits_{n\rightarrow\infty}dist(x_n,A_n)=0$. 
\end{lemma}
Since $G(u):=\|u\|_2^2-a^2$ is of class $C^1$ and for any $u\in S_a$, we have $\langle G'(u),u\rangle=2a^2>0$. Therefore, by the implicit function theorem, we know $S_a$ is a $C^{1}$-Finsler manifold.
\subsection{A compactness lemma for Palais-Smale sequences}
In this subsection, we will give a compactness lemma for the special Palais-Smale sequence, which satisfies a suitable additional condition.
\begin{lemma}\label{compact}\cite{Lions}
If $N\geq2$, then $H^2_{rad}(\mathbb{R}^N)$ is compactly embedding into $L^p(\mathbb{R}^N)$ for $2<p<\frac{2N}{(N-4)^+}$.
\end{lemma}
\begin{lemma}\label{com1}
Let $N\geq2$, $2<q<p<\frac{2N}{(N-4)^+}$, $\mu>0$, or $2<q\leq \bar{p}<p<\frac{2N}{(N-4)^+}$, $\mu<0$ and \eqref{con2} holds. If $\{u_n\}\subset S_{a,r}$ is a Palais-Smale sequence for $\mathcal{J}_{p,q}|_{S_{a,r}}$ at level $c\neq0$ and $P_{p,q}(u_n)\rightarrow 0$ as $n\rightarrow\infty$. Then there exists $u\in H^2_{rad}(\mathbb{R}^N)$ such that, up to a subsequence, $u_n\rightarrow u$ in $H^2(\mathbb{R}^N)$ as $n\rightarrow\infty$, and $u$ is a radial solution of  \eqref{question} with some $\lambda<0$.
\end{lemma}
\begin{proof}
We divide the proof into three steps.
\\
\textbf{Setp 1:} $\{u_n\}$ is bounded in $H^2(\mathbb{R}^N)$.

 By $P_{p,q}(u_n)\rightarrow 0$ as $n\rightarrow\infty$, we have
\begin{align}\label{bdd}
  c+1\geq \mathcal{J}_{p,q}(u_n)
  &=\Big(\frac{\gamma_p}{2}-\frac{1}{p}\Big)\int_{\mathbb{R}^N}|u_n|^pdx+\Big(\frac{\gamma_q}{2}-\frac{1}{q}\Big)\mu\int_{\mathbb{R}^N}|u_n|^qdx+o_n(1)
  \nonumber\\&=\Big(\frac{1}{2}-\frac{1}{q\gamma_q}\Big)\int_{\mathbb{R}^N}|\Delta u_n|^2dx+\Big(\frac{\gamma_p}{q\gamma_q}-\frac{1}{p}\Big)\int_{\mathbb{R}^N}|u_n|^pdx+o_n(1)\nonumber
  \\&=\Big(\frac{1}{2}-\frac{1}{p\gamma_p}\Big)\int_{\mathbb{R}^N}|\Delta u_n|^2dx+\Big(\frac{\gamma_q}{p\gamma_p}-\frac{1}{q}\Big)\mu\int_{\mathbb{R}^N}|u_n|^qdx+o_n(1).
\end{align}

(i) If $2< q\leq\bar{p}<p<\frac{2N}{(N-4)^+}$, $\mu>0$, by the H\"{o}lder inequality, there exists $\theta\in (0,1)$ such that
\begin{equation*}
  \|u_n\|_q^q\leq \|u_n\|_2^{q(1-\theta)}\|u_n\|_p^{q \theta}=a^{q(1-\theta)}\|u_n\|_p^{q \theta}.
\end{equation*}
Then by the first equation of \eqref{bdd}, we have
\begin{equation*}
  c+1\geq \Big(\frac{\gamma_p}{2}-\frac{1}{p}\Big)\|u_n\|_p^p+\Big(\frac{\gamma_q}{2}-\frac{1}{q}\Big)\mu a^{q(1-\theta)}\|u_n\|_p^{q \theta}.
\end{equation*}
Since $q\gamma_q\leq 2<\gamma_p$ and $q\theta <p$, we deduce that $\|u_n\|_p$ is bounded and hence $\|u_n\|_q$ is bounded. Using $P_{p,q}(u_n)\rightarrow 0$ as $n\rightarrow\infty$ again, we obtain the boundedness of $\|\Delta u_n\|_2$.

(ii) If $\bar{p}<q<p<\frac{2N}{(N-4)^+}$, $\mu>0$, by the second equation of \eqref{bdd} and $2<q\gamma_q<p\gamma_p$, we can deduce the boundedness of $\|\Delta u_n\|_2$.

(iii) If $2<q\leq \bar{p}<p<\frac{2N}{(N-4)^+}$, $\mu<0$, by the third equation of \eqref{bdd} and $q\gamma_q\leq 2<p\gamma_p$, we can deduce the boundedness of $\|\Delta u_n\|_2$.\\
\textbf{Setp 2:} $\{u_n\}$ has a nontrivial weak limit in $H^2(\mathbb{R}^N)$.

By using Lemma \ref{compact}, there exists $u\in H^2_{rad}(\mathbb{R}^N)$ such that, up to a subsequence, $u_n\rightharpoonup u$ in $H^2(\mathbb{R}^N)$, $u_n\rightarrow u$ in $L^p(\mathbb{R}^N)$ for $2<p<\frac{2N}{(N-4)^+}$ and a.e. in $\mathbb{R}^N$. Since $\{u_n\}\subset S_{a,r}$ is a Palais-Smale sequence for $\mathcal{J}_{p,q}|_{S_{a,r}}$, by the Lagrange multipliers rule, there exists a sequence $\{\lambda_n\}\subset \mathbb{R}$ such that
\begin{equation}\label{lag}
   \Delta^2u_n-\mu|u_n|^{q-2}u_n-|u_n|^{p-2}u_n=\lambda_n u_n+o_n(1).
\end{equation}
Testing the above equation with $u_n$, we obtain
\begin{equation*}
  \lambda_n a^2= \int_{\mathbb{R}^N}|\Delta u|^2dx-\int_{\mathbb{R}^N}|u|^pdx-\mu\int_{\mathbb{R}^N}|u|^qdx+o_n(1),
\end{equation*}
this with the boundedness of $\{u_n\}$ in $H^2(\mathbb{R}^N)\cap L^p(\mathbb{R}^N)\cap L^q(\mathbb{R}^N)$ yields that $\{\lambda_n\}$ is bounded in $\mathbb{R}$. Without loss of generality, we assume that $\lambda_n \rightarrow\lambda $ as $n\rightarrow\infty$.
Claim that $u\neq0$. Assume by contradiction that $u=0$, then $\|u_n\|_p,\|u_n\|_q\rightarrow0$, this together with $P_{p,q}(u_n)\rightarrow 0$ yields that $\|\Delta u_n\|_2\rightarrow0$, hence $\mathcal{J}_{p,q}(u_n)\rightarrow0$ as $n\rightarrow\infty$, which is contradict to $c\neq0$. Hence, $u_n\rightharpoonup u\neq0$ in $H^2(\mathbb{R}^N)$. From \eqref{lag}, we know $u$ is a radial solution of the equation
\begin{equation*}
  \Delta^2u=\lambda u+\mu|u|^{q-2}u+|u|^{p-2}u,\quad \text {in $\mathbb{R}^N$}.
\end{equation*}
 By \eqref{pohoim}, we can deduce that
\begin{equation}\label{nega}
   \lambda a^2=\frac{2N-p(N-4)}{N(2-p)}\|\Delta u\|_2^2+\frac{2\mu (p-q)}{q(2-p)}\|u\|_q^q.
\end{equation}

(i) If $\mu>0$, by $2<q<p<\frac{2N}{(N-4)^+}$, we obtain $\lambda<0$.

(ii) If $2<q\leq \bar{p}<p<\frac{2N}{(N-4)^+}$, $\mu<0$, and \eqref{con2} holds.
By Lemma \ref{poho} and Fatou lemma, we obtain
\begin{equation*}
\|\Delta u\|_2^2\leq \gamma_p\|u\|_p^{p}\leq \gamma_pB_{N,p}a ^{p(1-\gamma_p)}\|\Delta u\|_2^{p\gamma_p},
\end{equation*}
hence
\begin{equation*}
  \|\Delta u\|_2\geq \Big(\frac{1}{\gamma_pB_{N,p}a ^{p(1-\gamma_p)}}\Big)^{\frac{1}{p\gamma_p-2}}=\Big(\frac{4p a^{\frac{N(p-2)}{4}-p}}{N(p-2)B_{N,p}}\Big)^{\frac{4}{N(p-\bar{p})}}.
\end{equation*}
By \eqref{con2} and \eqref{nega}, we get
\begin{align*}
  \lambda a^2&\leq \frac{2N-p(N-4)}{N(2-p)}\|\Delta u\|_2^2+\frac{2\mu (p-q)}{q(2-p)}B_{N,q}a^{q(1-\gamma_q)}\|\Delta u\|_2^{q\gamma_q}\\
  &=\|\Delta u\|_2^{q\gamma_q}\Big[\frac{2N-p(N-4)}{N(2-p)}\|\Delta u\|_2^{2-q\gamma_q}+\frac{2\mu (p-q)}{q(2-p)}B_{N,q}a^{q(1-\gamma_q)}\Big]\\&
  \leq \|\Delta u\|_2^{q\gamma_q}\Big[\frac{2N-p(N-4)}{N(2-p)}\Big(\frac{4p a^{\frac{N(p-2)}{4}-p}}{N(p-2)B_{N,p}}\Big)^{\frac{4(2-q\gamma_q)}{N(p-\bar{p})}}+\frac{2\mu (p-q)}{q(2-p)}B_{N,q}a^{q(1-\gamma_q)}\Big]\\
  &=\|\Delta u\|_2^{q\gamma_q}\Big[\frac{2N-p(N-4)}{N(2-p)}\Big(\frac{4p a^{\frac{N(p-2)}{4}-p}}{N(p-2)B_{N,p}}\Big)^{\frac{\bar{p}-q}{p-\bar{p}}}+\frac{2\mu (p-q)}{q(2-p)}B_{N,q}a^{q-\frac{N(q-2)}{4}}\Big]<0,
\end{align*}
so $\lambda<0$.\\
\textbf{Setp 3:} $u_n\rightarrow u$ in $H^2(\mathbb{R}^N)$ as $n\rightarrow\infty$.

 Since $\lambda_n \rightarrow\lambda $ as $n\rightarrow\infty$, testing \eqref{question}, \eqref{lag} with $u_n-u$, and subtracting, we have
\begin{equation*}
  \langle\mathcal{J}'(u_n)-\mathcal{J}'(u),u_n-u\rangle-\lambda \int_{\mathbb{R}^N}|u_n-u|^2dx=o_n(1).
\end{equation*}
By the convergence of $\{u_n\}$ in $L^p(\mathbb{R}^N)\cap L^q(\mathbb{R}^N)$, we infer that
\begin{equation*}
\int_{\mathbb{R}^N}|\Delta (u_n-u)|^2dx+\lambda \int_{\mathbb{R}^N}|u_n-u|^2dx=o_n(1),
\end{equation*}
thus $\lambda<0$ implies that $u_n\rightarrow u$ in $H^2(\mathbb{R}^N)$ as $n\rightarrow\infty$.
\end{proof}
\begin{lemma}\label{com2}
Let $N\geq2$, $2<q<p<\frac{2N}{(N-4)^+}$, $\mu>0$, or $2<q\leq \bar{p}<p<\frac{2N}{(N-4)^+}$, $\mu<0$ and \eqref{con2} holds. If $\{u_n\}\subset S_{a}$ is a Palais-Smale sequence for $\mathcal{J}_{p,q}|_{S_{a}}$ at level $c\neq0$ and $P_{p,q}(u_n)\rightarrow 0$ as $n\rightarrow\infty$. Moreover, there exists a sequence $\{v_n\}\in S_{a,r}$ such that $\|u_n-v_n\|\rightarrow0$ as $n\rightarrow\infty$. Then there exists $u\in H^2_{rad}(\mathbb{R}^N)$ such that, up to a subsequence, $u_n\rightarrow u$ in $H^2(\mathbb{R}^N)$ as $n\rightarrow\infty$, and $u$ is a radial solution of \eqref{question} with some $\lambda<0$.
\end{lemma}
\begin{proof}
The proof is analogue to the proof of Lemma \ref{com1}. As in Lemma \ref{com1}, we can show that there exists $u\in H^2_{rad}(\mathbb{R}^N)$ such that, up to a subsequence, $v_n\rightarrow u$ in $H^2(\mathbb{R}^N)$. Since $\|u_n-v_n\|\rightarrow0$, the same convergence is inherited by $\{u_n\}$, thus $u_n\rightarrow u$ in $H^2(\mathbb{R}^N)$ as $n\rightarrow\infty$.
\end{proof}

\section{The case $\mu=0$}
In this section, we deal with the case $\mu=0$ and prove Theorems \ref{th1} and \ref{th7}.
\subsection{The case $N\geq2$, $2<p<\frac{2N}{(N-4)^+}$}
\begin{lemma}\label{nega1,}
For $2<p<\bar{p}$, we have $-\infty<m_{p,a}<0$.
\end{lemma}
\begin{proof}
By Lemma \ref{GN}, for any $u\in S_a$, we get
\begin{equation*}
  \mathcal{J}_{p}(u)\geq \frac{1}{2}\|\Delta u\|_2^2-\frac{B_{N,p}a^{p-\frac{N(p-2)}{4}}}{p}\|\Delta u\|_2^{\frac{N(p-2)}{4}}.
\end{equation*}
Since $2<p<\bar{p}$, we have $\frac{N(p-2)}{4}<2$, so $\mathcal{J}_{p}$ is coercive on $S_a$, which implies $m_{p,a}>-\infty$.

On the other hand, for any $u\in S_a$ and $\tau\rightarrow-\infty$, we have
\begin{equation*}
  \mathcal{J}_{p}(\tau*u)=\frac{e^{4\tau}}{2}\int_{\mathbb{R}^N}|\Delta u|^2dx-\frac{e^{\frac{N(p-2)}{2}\tau}}{p}\int_{\mathbb{R}^N}| u|^pdx\rightarrow 0^-.
\end{equation*}
Hence, $m_{p,a}\leq \mathcal{J}_{p}(\tau*u)<0$.
\end{proof}

\begin{lemma}\label{strict}
For $2<p<\bar{p}$, let $a_1,a_2>0$ be such that $a_1^2+a_2^2=a^2$, then $m_{p,a}<m_{p,a_1}+m_{p,a_2}$.
\end{lemma}
\begin{proof}
Let $c>0$, $\theta>1$ and $\{u_n\} \subset S_c$ be a minimizing sequence for $m_{p,c}$, then by $p>2$, we have
\begin{equation*}
  m_{p,\theta c}\leq \mathcal{J}_{p}( \theta u_n)=\frac{\theta^2}{2}\int_{\mathbb{R}^N}|\Delta u_n|^2dx-\frac{\theta^p}{2}\int_{\mathbb{R}^N}| u_n|^pdx< \theta ^2\mathcal{J}_{p}( u_n).
\end{equation*}
As a consequence, $m_{p,\theta c}\leq \theta^2 m_{p,c}$. The equality holds if and only if $\|u_n\|_p\rightarrow0$ as $n\rightarrow\infty$. But this is impossible, since otherwise we would find
\begin{equation*}
  0>m_{p,c}=\lim\limits_{n\rightarrow\infty}\mathcal{J}_{p}( u_n)\geq \lim\limits_{n\rightarrow\infty}\frac{1}{2}\int_{\mathbb{R}^N}|\Delta u_n|^2dx\geq0,
\end{equation*}
which is an absurd. So we have $m_{p,\theta c}< \theta^2 m_{p,c}$.

Without loss of generality, we assume that $a_1\geq a_2$.

(i) If $a_1=a_2$, then $m_{p,a}=m_{p,\sqrt{2}a_1}< 2m_{p,a_1}=m_{p,a_1}+m_{p,a_2}$.

(ii) If $a_1>a_2$, then
\begin{equation*}
m_{p,a}=m_{p,\frac{a}{a_1}a_1}< (\frac{a}{a_1})^2m_{p,a_1}=m_{p,a_1}+\frac{a^2-a_1^2}{a_1^2}m_{p,a_1}=m_{p,a_1}+\frac{a_2^2}{a_1^2}m_{p,a_1}<m_{p,a_1}+m_{p,a_2}.
\end{equation*}
\end{proof}

From Lemmas \ref{nega1,} and \ref{strict}, we immediately have the following corollary.
\begin{corollary}\label{corr}
For $2<p<\bar{p}$, $m_{p,a}$ is strictly decreasing in $a\in (0,\infty)$.
\end{corollary}

\begin{lemma}\label{lxf11}
If $\bar{p}<p<\frac{2N}{(N-4)^+}$, then
for $u\in S_a$, the function $\Psi_{u,0}$ has a unique critical point $t_u\in \mathbb{R}$ and $t_u$ is a strict maximum point at positive level. Moreover, $t_u*u\in \mathcal{P}_{p,a}=\mathcal{P}_{p,a}^-$ and $P_{p,a}(u)\leq0$ if and only if $t_u\leq0$.
\end{lemma}
\begin{proof}
It's easy to check that $\lim\limits_{\tau\rightarrow-\infty}\Psi_{u,0}(\tau)=0^+$ and $\lim\limits_{\tau\rightarrow\infty}\Psi_{u,0}(\tau)=-\infty$. Moreover, for any $u\in S_a$, $\Psi_{u,0}'(\tau)=0$ is and only if
\begin{equation*}
  \gamma_p e^{\frac{N(p-2)-8}{2}\tau}\|u\|_p^p=\| \Delta u\|_2^2.
\end{equation*}
Obviously, the above equation has at most one solution. Thus $\Psi_{u,0}$ has a unique critical point $t_u\in \mathbb{R}$ and $t_u$ is a strict maximum point at positive level. 
Moreover, $t_u*u\in \mathcal{P}_{p,a}$. Since $\Psi_{u,0}'(t)\leq0$ if and only if $t\geq t_u$, we have $P_p(u)=\frac{1}{2}\Psi_{u,0}'(0)\leq0$ if and only if $t_u\leq0$.

We claim that $\mathcal{P}_{p,a}^0=\emptyset$. Assume by contradiction that if there exists $u\in \mathcal{P}_{p,a}^0$, then
\begin{equation*}
 \|\Delta u\|_2^2=\gamma_p\|u\|_p^{p},\quad 2\|\Delta u\|_2^2=p\gamma_p^2\|u\|_p^p,
\end{equation*}
thus
\begin{equation*}
  \gamma_p(p\gamma_p-2)\|u\|_p^p=0.
\end{equation*}
By $p>\bar{p}$, we have $p\gamma_p>2$, hence $\|u\|_p=0$, which is contradict to $u\in S_a$. So $\mathcal{P}_{p,a}^0=\emptyset$. For any $u\in \mathcal{P}_{p,a}$, by the uniqueness of $t_u$ such that $t_u*u\in \mathcal{P}_{p,a}$, we know $t_u=0$. By $\Psi''_{u,0}(0)\leq 0$ and $\mathcal{P}_{p,a}^0=\emptyset$, we deduce that $\Psi''_{u,0}(0)< 0$, which implies that $u\in \mathcal{P}_{p,a}^-$, so that $\mathcal{P}_{p,a}=\mathcal{P}_{p,a}^-$.
\end{proof}
\begin{lemma}\label{lxf12}
It results that $M_{p,a}>0$.
\end{lemma}
\begin{proof}
If $u\in \mathcal{P}_{p,a}$, then by Lemmas \ref{GN} and \ref{poho}, we have
\begin{equation*}
  \|\Delta u\|_2^2=\gamma_p\|u\|_p^{p}\leq \gamma_p B_{N,p}a^{p-\frac{N(p-2)}{4}}\|\Delta u\|_2^{\frac{N(p-2)}{4}}.
\end{equation*}
Since $p>\bar{p}$, $\frac{N(p-2)}{4}>2$, there exists $\delta>0$ such that $\inf\limits_{u\in \mathcal{P}_{p,a}}\|\Delta u\|_2\geq \delta>0$. Using Lemma \ref{poho} again, we obtain
\begin{equation*}
  \mathcal{J}_{p}(u)=\Big(\frac{1}{2}-\frac{1}{p\gamma_p}\Big)\int_{\mathbb{R}^N}|\Delta u|^2dx\geq \Big(\frac{1}{2}-\frac{1}{p\gamma_p}\Big)\delta^2>0.
\end{equation*}
\end{proof}

\textbf{Proof of Theorem \ref{th1}:} For $(i)$, let $\{u_n\} \subset S_a$ be a minimizing sequence for $m_{p,a}$. Using Lemma \ref{Four}, if $2<p<\bar{p}$ and $p\in 2\mathbb{N}$, then $\{u_n^\sharp \}\in S_a$ and
\begin{equation*}
  \mathcal{J}_{p}(u_n^\sharp) \leq\mathcal{J}_{p}(u_n),
\end{equation*}
which implies that $\{u_n^\sharp\} $ is also a radial minimizing sequence for $m_{p,a}$. By Lemma \ref{nega1,}, we know $\mathcal{J}_{p}$ is coercive on $S_a$, so $\{u_n^\sharp\}$ is bounded in $H^2(\mathbb{R}^N)$. Then by Lemma \ref{compact}, there  exists $u\in H^2_{rad}(\mathbb{R}^N)$ such that, up to a subsequence, $u_n^\sharp\rightharpoonup u$ in $H^2(\mathbb{R}^N)$, $u_n^\sharp\rightarrow u$ in $L^p(\mathbb{R}^N)$ for $2<p<\frac{2N}{(N-4)^+}$ and a.e. in $\mathbb{R}^N$.
Using Fatou lemma and Lemma \ref{nega1,}, we obtain
\begin{equation*}
  m_{p,\|u\|_2}=\min\Big\{0,m_{p,\|u\|_2}\Big\}\leq\mathcal{J}_{p}(u)\leq \liminf\limits_{n\rightarrow\infty}\mathcal{J}_{p}(u_n^\sharp)=m_{p,a}<0, \quad \|u\|_2\leq a,
\end{equation*}
so $u\neq0$. By Corollary \ref{corr}, we can find that $m_{p,\|u\|_2}\geq m_{p,a}$, thus $\mathcal{J}_{p}(u)=m_{p,a}$. Using Corollary \ref{corr} again, we obtain $\|u\|_2=a$. Therefore, $m_{p,a}$ has a radial minimizer. By \eqref{pohoim} and $p<\frac{2N}{(N-4)^+}$, we deduce that
\begin{equation*}
  \lambda a^2=\frac{p(N-4)-2N}{4p}\|u\|_p^p<0,
\end{equation*}
so $\lambda<0$.

For $(a)$ of $(ii)$, by Lemma \ref{GN}, we have
\begin{equation}\label{B}
  B_{N,\bar{p}}=\frac{\bar{p}}{2\bar{a}^{\frac{8}{N}}}.
\end{equation}
Using Lemma \ref{GN} again, we get
\begin{equation*}
  \mathcal{J}_{p}(u)\geq \frac{1}{2}\|\Delta u\|_2^2-\frac{B_{N,\bar{p}}a^{\bar{p}-\frac{N(\bar{p}-2)}{4}}}{\bar{p}}\|\Delta u\|_2^{\frac{N(\bar{p}-2)}{4}}=\frac{1}{2}\Big(1-(\frac{a}{\bar{a}})^{\frac{8}{N}}\Big)\|\Delta u\|_2^2
\end{equation*}
for any $u\in S_a$. Hence, it results that $m_{p,a}\geq0$. Noted that $\mathcal{J}_{p}(\tau*u)\rightarrow0$ as $\tau\rightarrow-\infty$ for any $u\in S_a$, so $m_{p,a}=0$.

Assume by contradiction that  \eqref{question} has a solution $u$, then
by Lemmas \ref{GN}, \ref{poho} and \eqref{B}, we have
\begin{equation*}
 \|\Delta u\|_2^2=\frac{2}{\bar{p}}\|u\|_{\bar{p}}^{\bar{p}}\leq (\frac{a}{\bar{a}})^{\frac{8}{N}}\|\Delta u\|_2^2<\|\Delta u\|_2^2,
\end{equation*}
which is a contradiction.

For $(b)$ of $(ii)$, $m_{p,a}=0$ follows from a similar argument as $(a)$ of $(ii)$. By Lemma \ref{GN}, we have
\begin{equation*}
  B_{N,\bar{p}}=\frac{\|Q_{N,\bar{p}}\|_{\bar{p}}^{\bar{p}}}{\|Q_{N,\bar{p}}\|_2^{\bar{p}-\frac{N(\bar{p}-2)}{4}}\|\Delta Q_{N,\bar{p}}\|_2^{\frac{N(\bar{p}-2)}{4}}}=\frac{\|Q_{N,\bar{p}}\|_{\bar{p}}^{\bar{p}}}{\bar{a}^{\frac{8}{N}}\|\Delta Q_{N,\bar{p}}\|_2^{2}},
\end{equation*}
this together with \eqref{B} yields that $\|Q_{N,\bar{p}}\|_{\bar{p}}^{\bar{p}}=\frac{\bar{p}}{2}\|\Delta Q_{N,\bar{p}}\|_2^2$, hence $\mathcal{J}_{p}(Q_{N,\bar{p}})=0$. Namely, $m_{p,a}$ has a minimizer $Q_{N,\bar{p}}$, and $\lambda<0$. 

For $(c)$ of $(ii)$, let $u_a=\frac{a}{\bar{a}}Q_{N,\bar{p}}$, then $u_a\in S_a$ and
\begin{equation*}
  \mathcal{J}_{p}(u_a)=\frac{(\frac{a}{\bar{a}})^2}{2}\int_{\mathbb{R}^N}|\Delta Q_{N,\bar{p}}|^2dx-\frac{(\frac{a}{\bar{a}})^p}{2}\int_{\mathbb{R}^N}| Q_{N,\bar{p}}|^pdx<(\frac{a}{\bar{a}})^2\mathcal{J}_{p}(Q_{N,\bar{p}})=0.
\end{equation*}
Thus, for any $\tau\rightarrow+\infty$,
\begin{equation*}
  \mathcal{J}_{p}(\tau*u_a)=e^{4\tau}\mathcal{J}_{p}(u_a)\rightarrow-\infty,
\end{equation*}
which implies that $m_{p,a}=-\infty$.

For $(iii)$, since $p>\bar{p}$, for any $u\in S_a$, it holds that $\mathcal{J}_{p}(u)\rightarrow-\infty$ as $\tau\rightarrow+\infty$. Thus, $m_{p,a}=-\infty$.
\qed

\subsection{The case $N\geq9$, $p=4^*$}
\begin{lemma}\label{critical11}
For any $u\in S_a$, the function $\Psi_{u,0}$ has a unique critical point $t_u\in \mathbb{R}$ and $t_u$ is a strict maximum point at positive level. Moreover, $t_u*u\in \mathcal{P}_{4^*,a}=\mathcal{P}_{4^*,a}^-$, and $e^{t_u}=\Big(\frac{\|\Delta u\|_2^2}{\|u\|_{4^*}^{4^*}}\Big)^{\frac{1}{2(4^*-2)}}$. 
\end{lemma}
\begin{proof}
The proof is similar to the proof of Lemma \ref{lxf11}, we omit it.
\end{proof}

\begin{lemma}\label{critical12}
$M_{4^*,a}=\inf\limits_{u\in S_a}\max \limits_{\tau \in \mathbb{R}}\mathcal{J}_{4^*}(\tau*u)$.
\end{lemma}
\begin{proof}
For any $u\in \mathcal{P}_{4^*,a}$, by Lemma \ref{critical11}, we have $t_u=0$ and
\begin{equation*}
  \mathcal{J}_{4^*}(u)=\max \limits_{\tau \in \mathbb{R}}\mathcal{J}_{4^*}(\tau*u)\geq \inf\limits_{u\in S_a}\max \limits_{\tau \in \mathbb{R}}\mathcal{J}_{4^*}(\tau*u).
\end{equation*}
By the arbitrariness of $u\in \mathcal{P}_{4^*,a}$, we obtain $M_{4^*,a}\geq\inf\limits_{u\in S_a}\max \limits_{\tau \in \mathbb{R}}\mathcal{J}_{4^*}(\tau*u)$.
On the other hand, for any $u\in S_a$, we have $t_u*u\in \mathcal{P}_{4^*,a}$, and
\begin{equation*}
   M_{4^*,a}\leq \mathcal{J}_{4^*}(t_u*u)=\max \limits_{\tau \in \mathbb{R}}\mathcal{J}_{4^*}(\tau*u).
\end{equation*}
By the arbitrariness of $u\in S_a$, $ M_{4^*,a}\leq \inf\limits_{u\in S_a}\max \limits_{\tau \in \mathbb{R}}\mathcal{J}_{4^*}(\tau*u)$.
\end{proof}

\textbf{Proof of Theorem \ref{th7}:} For any $u\in S_a$, by Lemmas \ref{critical11} and \ref{critical12}, we have
\begin{align*}
  M_{4^*,a}&=\inf\limits_{u\in S_a}\max \limits_{\tau \in \mathbb{R}}\mathcal{J}_{4^*}(\tau*u)\\&
  =\inf\limits_{u\in S_a}\Big[\frac{1}{2}\Big(\frac{\|\Delta u\|_2^2}{\|u\|_{4^*}^{4^*}}\Big)^{\frac{2}{4^*-2}}\| \Delta u\|_2^2-\frac{1}{4^*}\Big(\frac{\|\Delta u\|_2^2}{\|u\|_{4^*}^{4^*}}\Big)^{\frac{\cdot4^*}{4^*-2}}\|u\|_{4^*}^{4^*}\Big]\\
  &=\inf\limits_{u\in S_a}\frac{2}{N}\Big(\frac{\|\Delta u\|_2^2}{\|u\|_{4^*}^{2}}\Big)^{\frac{4^*}{4^*-2}}=\inf\limits_{u\in H^2(\mathbb{R}^N)\backslash \{0\}}\frac{2}{N}\Big(\frac{\|\Delta u\|_2^2}{\|u\|_{4^*}^{2}}\Big)^{\frac{N}{4}},
\end{align*}
that is, $M_{4^*,a}$ is equivalent to the minimization of the Sobolev quotient in $H^2(\mathbb{R}^N)\backslash \{0\}$. By density of $H^2(\mathbb{R}^N)$ in $D^{2,2}(\mathbb{R}^N)$, we infer that $M_{4^*,a}=\frac{2}{N}S^{\frac{N}{4}}$, and $M_{4^*,a}$ is achieved if and only if the functions $U_{\varepsilon,0}$ defined in \eqref{deU} stay in $L^2(\mathbb{R}^N)$, namely, if and only if $N\geq9$.

By Lemma \ref{natural}, we know $\mathcal{P}_{4^*,a}$ is a natural constraint, thus by the Lagrange multiplies rule, there exists $\lambda\in \mathbb{R}$ such that
\begin{equation*}
  \Delta^2 U_{\varepsilon,0}=\lambda U_{\varepsilon,0}+|U_{\varepsilon,0}|^{4^*-2}U_{\varepsilon,0},\quad x\in \mathbb{R}^N.
\end{equation*}
By using $P_{4^*}(U_{\varepsilon,0})=0$, we obtain
$\lambda=0$.
\qed
\section{The case $\mu\neq0$}
In this section, we deal with the case $\mu\neq0$ and prove Theorems \ref{th2}-\ref{th8}.
\subsection{The case $N\geq2$, $2<q<p<\bar{p}$ and $\mu>0$}

With a similar argument of $(i)$ in Theorem \ref{th1}, we have the following lemmas.
\begin{lemma}\label{nega2}
For $2<q<p<\bar{p}$, $\mu>0$, $\mathcal{J}_{p,q}$ is coercive on $S_a$ and $-\infty<m_{p,q,a}<0$.
\end{lemma}

\begin{lemma}\label{strict2}
For $2<q<p<\bar{p}$, $\mu>0$, let $a_1,a_2>0$ be such that $a_1^2+a_2^2=a^2$, then $m_{p,q,a}<m_{p,q,a_1}+m_{p,q,a_2}$.
\end{lemma}

\begin{corollary}\label{corr2}
For $2<q<p<\bar{p}$, $\mu>0$, $m_{p,q,a}$ is strictly decreasing in $a>0$.
\end{corollary}

\textbf{Proof of Theorem \ref{th2}:} By Lemma \ref{nega2} and Corollary \ref{corr2}, we can adopt a similar way as the proof of $(i)$ in Theorem \ref{th1} to complete the proof.
Since $q<p<\frac{2N}{(N-4)^+}$, $\mu>0$ and $u\neq0$, using \eqref{pohoim}, we have
\begin{equation*}
  \lambda a^2=\frac{\mu[q(N-4)-2N]}{4q}\|u\|_q^q+\frac{p(N-4)-2N}{4p}\|u\|_p^p,
\end{equation*}
so $\lambda<0$.
\qed

\subsection{The case $N\geq2$, $2<q<p=\bar{p}$ and $\mu\neq0$}

\begin{lemma}\label{nega3}
For $2<q<p=\bar{p}$ and $0<a<\bar{a}$, if $\mu>0$, then $-\infty<m_{p,q,a}<0$; If $\mu<0$, then $m_{p,q,a}=0$.
\end{lemma}
\begin{proof}
By Lemma \ref{GN} and \eqref{B}, for any $u\in S_a$, we get
\begin{equation*}
  \mathcal{J}_{p,q}(u)\geq \Big(1-(\frac{a}{\bar{a}})^{\frac{8}{N}}\Big)\|\Delta u\|_2^2-\frac{\mu B_{N,q}a^{q-\frac{N(q-2)}{4}}}{q}\|\Delta u\|_2^{\frac{N(q-2)}{4}}.
\end{equation*}
Since $q<\bar{p}$, $\frac{N(q-2)}{4}<2$, if $\mu>0$, we know $\mathcal{J}_{p,q}$ is coercive on $S_a$, which implies that $m_{p,q,a}>-\infty$. If $\mu<0$, then it holds that
$m_{p,q,a}\geq0$.

On the other hand, for any $u\in S_a$ and $\tau\rightarrow-\infty$, if $\mu>0$, we have
\begin{equation*}
  \mathcal{J}_{p,q}(\tau*u)=\frac{e^{4\tau}}{2}\int_{\mathbb{R}^N}|\Delta u|^2dx-\frac{e^{4\tau}}{p}\int_{\mathbb{R}^N}| u|^pdx-\frac{\mu e^{\frac{N(q-2)}{2}\tau}}{q}\int_{\mathbb{R}^N}| u|^qdx\rightarrow 0^-.
\end{equation*}
Hence, $m_{p,q,a}\leq \mathcal{J}_{p,q}(\tau*u)<0$. While, if $\mu<0$, $\mathcal{J}_{p,q}(\tau*u)\rightarrow 0^+$, thus $m_{p,q,a}=0$.
\end{proof}
Follow a similar argument of $(i)$ in Theorem \ref{th1}, we have
\begin{lemma}\label{strict3}
For $2<q<p=\bar{p}$, $\mu>0$, let $a_1,a_2>0$ be such that $a_1^2+a_2^2=a^2<\bar{a}^2$, then $m_{p,q,a}<m_{p,q,a_1}+m_{p,q,a_2}$.
\end{lemma}

\begin{corollary}\label{corr3}
For $2<q<p=\bar{p}$, $\mu>0$, $m_{p,q,a}$ is strictly decreasing in $0<a<\bar{a}$.
\end{corollary}

\textbf{Proof of Theorem \ref{th3}:} For $(a)$ of $(i)$, the proof is similar to the proof of $(i)$ in Theorem \ref{th1}, so we omit it here. 

For $(b)$ of $(i)$, from Lemma \ref{nega3}, we have $m_{p,q,a}=0$. Assume by contradiction that \eqref{question} has a solution $u$, then
by Lemma \ref{poho}, we have
\begin{equation*}
 \|\Delta u\|_2^2=\frac{2}{\bar{p}}\|u\|_{\bar{p}}^{\bar{p}}+\mu \gamma_q\|u\|_q^q.
\end{equation*}
Recall that $m_{p,a}=0$ from $(a)$ of $(ii)$ in Theorem \ref{th1}, hence
\begin{equation*}
  0>\mu \gamma_q\|u\|_q^q=\|\Delta u\|_2^2-\frac{2}{\bar{p}}\|u\|_{\bar{p}}^{\bar{p}}=2\mathcal{J}_{p}(u)\geq 2m_{p,a}=0
\end{equation*}
which is a contradiction.

For $(a)$ of $(ii)$, according to $(b)$ of $(ii)$ in Theorem \ref{th1}, there exists $Q_{N,\bar{p}}\in S_{\bar{a}}$ such that $\mathcal{J}_{p}(Q_{N,\bar{p}})=m_{p,a}=0$. Hence, for any $\tau\in \mathbb{R}$,
\begin{equation*}
  \mathcal{J}_{p,q}(\tau*Q_{N,\bar{p}})=e^{4\tau}\mathcal{J}_{p}(Q_{N,\bar{p}})-\frac{\mu e^{\frac{N(q-2)}{2}\tau}}{q}\int_{\mathbb{R}^N}| Q_{N,\bar{p}}|^{q}dx=-\frac{\mu e^{\frac{N(q-2)}{2}\tau}}{q}\int_{\mathbb{R}^N}| Q_{N,\bar{p}}|^{q}dx.
\end{equation*}
Since $\mu>0$, we have $\mathcal{J}_{p,q}(\tau*Q_{N,\bar{p}})\rightarrow-\infty$ as $\tau\rightarrow\infty$, so $m_{p,q,a}=-\infty$.

For $(b)$ of $(ii)$, from Lemma \ref{nega3}, it's easy to see
$m_{p,q,a}=0$. 
The rest of the proof is similar to the one of $(b)$ of $(i)$, we omit it.

For $(iii)$, 
let $u_a=\frac{a}{\bar{a}}Q_{N,\bar{p}}$, from the proof of $(c)$ of $(ii)$ in Theorem \ref{th1}, we have $u_a\in S_a$ and $\mathcal{J}_{p}(u_a)<0$.
Thus, for any $\tau\rightarrow+\infty$, by $q<\bar{p}$, $\frac{N(q-2)}{2}<4$, we obtain
\begin{equation*}
  \mathcal{J}_{p,q}(\tau*u_a)=e^{4\tau}\mathcal{J}_{p}(u_a)-\frac{\mu e^{\frac{N(q-2)}{2}\tau}}{q}\int_{\mathbb{R}^N}| u_a|^{q}dx\rightarrow-\infty,
\end{equation*}
which implies that $m_{p,q,a}=-\infty$.
\qed

\subsection{The case $N\geq2$, $2<q<\bar{p}<p<\frac{2N}{(N-4)^+}$ and  $\mu>0$}\label{mountain}

By Lemma \ref{GN}, for any $u\in S_a$, we get
\begin{equation*}
  \mathcal{J}_{p,q}(u)\geq \frac{1}{2}\|\Delta u\|_2^2-\frac{B_{N,p}a^{p-\frac{N(p-2)}{4}}}{p}\|\Delta u\|_2^{\frac{N(p-2)}{4}}-\frac{B_{N,q}a^{q-\frac{N(q-2)}{4}}}{q}\|\Delta u\|_2^{\frac{N(q-2)}{4}}.
\end{equation*}
To understand the geometry of the functional $ \mathcal{J}_{p,q}|_{S_a}$, we consider the function $h:\mathbb{R}^+\rightarrow \mathbb{R}$
\begin{equation*}
  h(t):=\frac{1}{2}t-\frac{B_{N,p}a^{p-\frac{N(p-2)}{4}}}{p}t^{\frac{N(p-2)}{8}}-\frac{\mu B_{N,q}a^{q-\frac{N(q-2)}{4}}}{q}t^{\frac{N(q-2)}{8}}.
\end{equation*}
Since $\mu>0$ and $\frac{N(q-2)}{8}<1<\frac{N(p-2)}{8}$, we have $\lim\limits_{t\rightarrow0^+}h(t)=0^-$ and $\lim\limits_{t\rightarrow\infty}h(t)=-\infty$.

\begin{lemma}\label{sx1}
Let $\mu, a>0$ satisfies \eqref{con1}, then the function $h$ has a local strict minimum at negative level and a global strict maximum at positive level. Moreover, there exists $0<R_0<R_1$, both depending on $\mu$ and $a$, such that $h(R_0)=h(R_1)=0$ and $h(t)>0$ if and only if $t\in (R_0,R_1)$.
\end{lemma}
\begin{proof}
Since
\begin{equation*}
  h(t)=t^{\frac{N(q-2)}{8}}\Big[\frac{1}{2}t^{\frac{N(\bar{p}-q)}{8}}-\frac{B_{N,p}a^{p-\frac{N(p-2)}{4}}}{p}t^{\frac{N(p-q)}{8}}-\frac{\mu B_{N,q}a^{q-\frac{N(q-2)}{4}}}{q}\Big]
\end{equation*}
for $t>0$, we have $h(t)>0$ if and only if
\begin{equation*}
  \varphi(t)>\frac{\mu B_{N,q}a^{q-\frac{N(q-2)}{4}}}{q},\quad \text{with}\quad \varphi(t):=\frac{1}{2}t^{\frac{N(\bar{p}-q)}{8}}-\frac{B_{N,p}a^{p-\frac{N(p-2)}{4}}}{p}t^{\frac{N(p-q)}{8}}.
\end{equation*}
It is easy to check that $\varphi(t)$ has a unique critical point on $(0,\infty)$, which is a global maximum point at positive level in
\begin{equation*}
  \bar{t}:=\Big(\frac{p(\bar{p}-q)}{2B_{N,p}(p-q)}\Big)^{\frac{8}{N(p-\bar{p})}}a^{\frac{4p-p\bar{p}-4}{p-\bar{p}}},
\end{equation*}
and the maximum level is
\begin{equation*}
  \varphi(\bar{t})=\frac{p-\bar{p}}{2(p-q)}\Big(\frac{p(\bar{p}-q)}{2B_{N,p}(p-q)}\Big)^{\frac{\bar{p}-q}{p-\bar{p}}}a^{\frac{N(4p-p\bar{p}-4)(\bar{p}-q)}{8(p-\bar{p})}}.
\end{equation*}
Therefore, $h(t)>0$ on an open internal $(R_0,R_1)$ if and only if $\varphi(\bar{t})>\frac{\mu B_{N,q}a^{q-\frac{N(q-2)}{4}}}{q}$, which is guaranteed by \eqref{con1}. It follows immediately that $h$ has a global strict maximum at positive level in $(R_0,R_1)$. Moreover, by $\lim\limits_{t\rightarrow0^+}h(t)=0^-$, there exists a local strict minimum at negative level in $(0,R_0)$. The fact that $h$ has no other critical points can be verified by $h'(t)=0$ if and only if
\begin{equation*}
  \psi(t)=\frac{\mu N B_{N,q}(q-2)a^{q-\frac{N(q-2)}{4}}}{4q}>0,\quad \text{with}\quad\psi(t):=t^{\frac{N(\bar{p}-p)}{8}}-\frac{NB_{N,p}(p-2)a^{p-\frac{N(p-2)}{4}}}{4p}t^{\frac{N(p-q)}{8}}.
\end{equation*}
Since $0<\frac{N(\bar{p}-q)}{8}<\frac{N(p-q)}{8}$, the above equation has at most two solutions. This ends the proof.
\end{proof}

\begin{lemma}\label{sx2}
$\mathcal{P}_{p,q,a}^0=\emptyset$ and $\mathcal{P}_{p,q,a}$ is a smooth manifold of codimension $2$ in $H^2(\mathbb{R}^N)$.
\end{lemma}
\begin{proof}
Assume that $u\in \mathcal{P}_{p,q,a}^0$, then by $P_{p,q}(u)=0$ and $\Psi_u''(0)=0$, we deduce that
\begin{equation*}
  (p\gamma_p-2)\gamma_p\|u\|_p^p=(2-q\gamma_q)\mu\gamma_q\|u\|_q^q.
\end{equation*}
Using $P_{p,q}(u)=0$ again, we have
\begin{equation*}
  \|\Delta u\|_2^2=\frac{p\gamma_p-q\gamma_q}{2-q\gamma_q}\gamma_p\|u\|_p^p\leq \frac{p\gamma_p-q\gamma_q}{2-q\gamma_q}\gamma_pB_{N,p}a^{p(1-\gamma_p)}\|\Delta u\|_2^{p\gamma_p}
\end{equation*}
and
\begin{equation*}
  \|\Delta u\|_2^2=\frac{p\gamma_p-q\gamma_q}{p\gamma_p-2}\mu\gamma_q\|u\|_q^q\leq \frac{p\gamma_p-q\gamma_q}{p\gamma_p-2}\mu\gamma_qB_{N,q}a^{q(1-\gamma_q)}\|\Delta u\|_2^{q\gamma_q}.
\end{equation*}
Hence
\begin{equation*}
  \Big(\frac{2-q\gamma_q}{B_{N,p}\gamma_p(p\gamma_p-q\gamma_q)}\Big)^{\frac{1}{p\gamma_p-2}}a^{\frac{p(\gamma_p-1)}{p\gamma_p-2}}\leq \Big(\frac{p\gamma_p-2}{B_{N,q}\gamma_q(p\gamma_p-q\gamma_q)}\Big)^{\frac{1}{q\gamma_q-2}}\Big(\mu a^{q(1-\gamma_q)}\Big)^{\frac{1}{2-q\gamma_q}},
\end{equation*}
that is
\begin{equation*}
  \mu a^{\gamma(p,q)}\geq\Big(\frac{\bar{p}-q}{B_{N,p}\gamma_p(p-q)}\Big)^{\frac{\bar{p}-q}{p-\bar{p}}}\Big(\frac{p-\bar{p}}{B_{N,q}\gamma_q(p-q)}\Big),
\end{equation*}
where $\gamma(p,q)$ is defined in \eqref{gamma}.
Since
\begin{equation*}
  \frac{q\gamma_q}{2}\Big(\frac{p\gamma_p}{2}\Big)^{\frac{\bar{p}-q}{p-\bar{p}}}\leq1,
\end{equation*}
we obtain
\begin{equation*}
  \mu a^{\gamma(p,q)}\geq\Big(\frac{p(\bar{p}-q)}{2B_{N,p}(p-q)}\Big)^{\frac{\bar{p}-q}{p-\bar{p}}}\Big(\frac{q(p-\bar{p})}{2B_{N,q}(p-q)}\Big)
\end{equation*}
which is contradict to \eqref{con1}. This prove that $\mathcal{P}_{p,q,a}^0=\emptyset$.

Next, we check that $\mathcal{P}_{p,q,a}$ is a smooth manifold of codimension $2$ in $H^2(\mathbb{R}^N)$. Noted that $\mathcal{P}_{p,q,a}=\{u\in H^2(\mathbb{R}^N):G(u)=0,\,P(u)=0\}$ for $G(u)=\|u\|_2^2-a^2$, with $P$ and $G$ of class $C^1$ in $H^2(\mathbb{R}^N)$. Thus, we have to show that the differential $(dG(u),dP(u)):H^2(\mathbb{R}^N)\rightarrow \mathbb{R}^2$ is surjective for any $u\in \mathcal{P}_{p,q,a}$. To this end, we prove that for any $u\in \mathcal{P}_{p,q,a}$, there exists $\varphi\in T_uS_a$ such that $dP(u)[\varphi]\neq0$. Once that the existence of $\varphi$ is established, the system
\begin{align*}
\begin{aligned}
  \begin{split}
  \left\{
  \begin{array}{ll}
   dG(u)[\alpha \varphi+\beta u]=x \\
    dP(u)[\alpha \varphi+\beta u]=y
    \end{array}
    \right.
    \Leftrightarrow
    \left\{
  \begin{array}{ll}
   \beta a^2=x \\
    \alpha dP(u)[\varphi]+\beta dP(u)[u]=y
    \end{array}
    \right.
  \end{split}
  \end{aligned}
  \end{align*}
is solvable with respect to $\alpha,\beta$ for any $(x,y)\in \mathbb{R}^2$, and hence the surjectivity is proved.

Now, suppose by contradiction that:
\begin{equation*}
  \text{There exists $u\in \mathcal{P}_{p,q,a}$ such that $ dP(u)[\varphi]=0$ for any  $\varphi\in T_uS_a$.}
\end{equation*}
Then $u$ is a constraint critical point for the functional $P_{p,q}$ on $S_a$, and hence by the Lagrange multiplies rule, there exists $\nu\in \mathbb{R}$ such that
\begin{equation*}
  P_{p,q}'=\lambda u,\quad x\in \mathbb{R}^N,
\end{equation*}
i.e.,
\begin{equation*}
  \Delta^2u=\lambda u+\frac{\mu q\gamma_q}{2}|u|^{q-2}u+\frac{p\gamma_p}{2}|u|^{p-2}u,\quad x\in \mathbb{R}^N.
\end{equation*}
By Lemma \ref{poho}, we have
\begin{equation*}
  \|\Delta u\|_2^2=\frac{\mu q\gamma_q^2}{2}\|u\|_q^{q}+\frac{p\gamma_p^2}{2}\|u\|_p^{p},
\end{equation*}
which means $u\in \mathcal{P}_{p,q,a}^0$, a contradiction.
\end{proof}
\begin{lemma}\label{sx3}
For $u\in S_a$, the function $\Psi_u$ has exactly two critical points $s_u<t_u\in \mathbb{R}$ and two zeros $c_u<d_u\in \mathbb{R}$, with $s_u<c_u<t_u<d_u$. Moreover:
\\(i) $s_u*u\in \mathcal{P}_{p,q,a}^+$ and $t_u*u\in \mathcal{P}_{p,q,a}^-$. If $\tau*u\in \mathcal{P}_{p,q,a}$, then either $\tau=s_u$ or $\tau=t_u$.\\
(ii) For any $\tau\leq c_u$, $\|\Delta (\tau*u)\|_2^2\leq R_0$, and
\begin{equation*}
  \mathcal{J}_{p,q}(s_u*u)=\min\Big\{\mathcal{J}_{p,q}(\tau*u):\tau\in \mathbb{R} \quad\text{and}\quad \|\Delta (\tau*u)\|_2^2< R_0\Big\}<0.
\end{equation*}
(iii) We have \begin{equation*}
  \mathcal{J}_{p,q}(t_u*u)=\max\Big\{\mathcal{J}_{p,q}(\tau*u):\tau\in \mathbb{R} \Big\}>0,
\end{equation*}
and $\Psi_u$ is strictly decreasing and concave on $(t_u,\infty)$.\\
(iv)
The maps $u\in S_a\mapsto s_u\in \mathbb{R}$ and $u\in S_a\mapsto t_u\in \mathbb{R}$ are of class $C^1$.
\end{lemma}
\begin{proof}
Recall that
\begin{equation*}
  \Psi_u(\tau)= \mathcal{J}_{p,q}(\tau*u)\geq h(\|\Delta (\tau*u)\|_2^2)=h(e^{4\tau}\|\Delta u\|_2^2),
\end{equation*}
thus
$\Psi_u$ is positive on $(C(R_0),C(R_1))$, where
\begin{equation*}
C(R_0):=\frac{1}{4}\ln \frac{R_0}{\|\Delta u\|_2^2},\quad C(R_0):=\frac{1}{4}\ln \frac{R_1}{\|\Delta u\|_2^2}.
\end{equation*}
It is clear that $\lim\limits_{\tau\rightarrow-\infty}\Psi_u(\tau)=0^-$ and $\lim\limits_{\tau\rightarrow\infty}\Psi_u(\tau)=-\infty$, thus $\Psi_u$ has at least two critical points. In fact, $\Psi_u'(\tau)=0$ is and only if
\begin{equation*}
  \varphi(\tau)=\mu \gamma_q \|u\|^q>0,\quad \text{with}\quad\varphi(\tau):=e^{\frac{N(\bar{p}-q)}{2}\tau}\| \Delta u\|_2^2-\gamma_p e^{\frac{N(p-q)}{2}\tau}\|u\|_p^p.
\end{equation*}
By $0<\frac{N(\bar{p}-q)\tau}{2}<\frac{N(p-q)\tau}{2}$, we know the above equation has at most two solutions. Thus $\Psi_u$ has exactly two critical points $s_u<t_u$,
 with $s_u$ a local minimum point on $(0,C(R_0))$ at negative level, and $t_u$ a global maximum point at positive level.

 By Lemma \ref{onP}, we know $\tau*u\in \mathcal{P}_{p,q,a}$ implies $\tau=s_u$ or $\tau=t_u$.
From $\Psi_u''(s_u)\geq0$, $\Psi_u''(t_u)\leq0$, and $\mathcal{P}_{p,q,a}^0=\emptyset$, we have $s_u*u\in \mathcal{P}_{p,q,a}^+$ and $t_u*u\in \mathcal{P}_{p,q,a}^-$. Moreover, $\Psi_u$ is strictly decreasing and concave on $(t_u,\infty)$.
For any $\tau\leq c_u\leq C(R_0)$, $\|\Delta (\tau*u)\|_2^2=e^{4\tau}\|\Delta u\|_2^2\leq R_0$. It remains to show that $u\in S_a\mapsto s_u\in \mathbb{R}$ and $u\in S_a\mapsto t_u\in \mathbb{R}$ are of class $C^1$. To this end, we apply the implicit function theorem on the $C^1$ function $\Phi(u,\tau):=\Psi_u'(\tau)$. Then $\Phi(u,s_u)=0$, $\partial_\tau \Phi(u,s_u)=\Psi_u''(s_u)>0$. Since $\mathcal{P}_{p,q,a}^0=\emptyset$, it is not possible to pass with continuity from $\mathcal{P}_{p,q,a}^+$ to $\mathcal{P}_{p,q,a}^-$, then $u\in S_a\mapsto s_u\in \mathbb{R}$ is of class $C^1$. The same arguments proves that $u\in S_a\mapsto t_u\in \mathbb{R}$ is of class $C^1$.
\end{proof}

For any $k>0$, set
\begin{equation*}
  \mathcal{A}_k:=\Big\{u\in S_a:\|\Delta u\|_2^2<k\},
\end{equation*}
and
\begin{equation}\label{jubu}
c_{p,q,a}:=\inf\limits_{u\in \mathcal{A}_{R_0}}\mathcal{J}_{p,q}.
\end{equation}

\begin{lemma}\label{sx4}
$ \mathcal{P}_{p,q,a}^+\subset \mathcal{A}_{R_0}$, and $\sup\limits_{u\in \mathcal{P}_{p,q,a}^+}\mathcal{J}_{p,q}(u)\leq0\leq \inf\limits_{u\in \mathcal{P}_{p,q,a}^-}\mathcal{J}_{p,q}(u)$.
\end{lemma}
\begin{proof}
For any fixed $u\in S_a$, by Lemma \ref{sx3}, we have $s_u<c_u\leq C(R_0)$, so
\begin{equation*}
  \|\Delta (s_u*u)\|_2^2=e^{4s_u}\|\Delta u\|_2^2<R_0.
\end{equation*}
By the arbitrariness of $u\in S_a$ and $s_u*u\in \mathcal{P}_{p,q,a}^+$, we have $\mathcal{P}_{p,q,a}^+\subset \mathcal{A}_{R_0}$. Using Lemma \ref{sx3} again, we immediately know $\sup\limits_{u\in \mathcal{P}_{p,q,a}^+}\mathcal{J}_{p,q}(u)\leq0\leq \inf\limits_{u\in \mathcal{P}_{p,q,a}^-}\mathcal{J}_{p,q}(u)$.
\end{proof}

\begin{lemma}\label{sx5}
It results that $-\infty<c_{p,q,a}<0$ and
\begin{equation*}
  c_{p,q,a}=\inf\limits_{u\in \mathcal{P}_{p,q,a}^+}\mathcal{J}_{p,q}(u)=M_{p,q,a}\quad \text{and}\quad c_{p,q,a}<\inf\limits_{u\in \overline{\mathcal{A}_{R_0}}\backslash \mathcal{A}_{R_0-\rho}}\mathcal{J}_{p,q}(u)
\end{equation*}
for $\rho>0$ small enough.
\end{lemma}
\begin{proof}
For $u\in \mathcal{A}_{R_0}$, we have
\begin{equation*}
  \mathcal{J}_{p,q}(u)\geq h(\|\Delta u\|_2^2)\geq \min\limits_{t\in [0,R_0]}h(t)>-\infty,
\end{equation*}
so $c_{p,q,a}>-\infty$. Moreover, for any $\tau<<-1$ small enough, $\|\Delta (\tau*u)\|_2^2<R_0$ and $\mathcal{J}_{p,q}(\tau*u)<0$, hence $c_{p,q,a}<0$.

From lemmas \ref{sx2} and \ref{sx4}, it's easy to get $\inf\limits_{u\in \mathcal{P}_{p,q,a}^+}\mathcal{J}_{p,q}(u)=M_{p,q,a}$. Next, we prove $c_{p,q,a}=\inf\limits_{u\in \mathcal{P}_{p,q,a}^+}\mathcal{J}_{p,q}(u)$. By Lemma \ref{sx4}, we know $ \mathcal{P}_{p,q,a}^+\subset \mathcal{A}_{R_0}$, thus $c_{p,q,a}\leq\inf\limits_{u\in \mathcal{P}_{p,q,a}^+}\mathcal{J}_{p,q}(u)$. On the other hand, for any $u\in \mathcal{A}_{R_0}$, we have $s_u*u\in \mathcal{P}_{p,q,a}^+$ and
\begin{equation*}
  \mathcal{J}_{p,q}(s_u*u)=\min\Big\{\mathcal{J}_{p,q}(\tau*u):\tau\in \mathbb{R} \quad\text{and}\quad \|\Delta (\tau*u)\|_2^2< R_0\Big\}\leq \mathcal{J}_{p,q}(u).
\end{equation*}
By the arbitrariness of $u\in \mathcal{A}_{R_0}$, we obtain $\inf\limits_{u\in \mathcal{P}_{p,q,a}^+}\mathcal{J}_{p,q}(u)\leq \mathcal{J}_{p,q}(s_u*u)\leq c_{p,q,a}$.

By the continuity of $h$, there exists $\rho>0$ such that $h(t)\geq \frac{c_{p,q,a}}{2}$ for any $t\in [R_0-\rho,R_0]$. Therefore, for any $u\in \overline{\mathcal{A}_{R_0}}\backslash \mathcal{A}_{R_0-\rho}$, we have
\begin{equation*}
\mathcal{J}_{p,q}(u)\geq h(\|\Delta u\|_2^2)\geq \frac{c_{p,q,a}}{2}>c_{p,q,a}.
\end{equation*}
\end{proof}

\begin{theorem}\label{local}
$c_{p,q,a}$ can be achieved by some $\hat{u}\in S_{a,r}$, and $\hat{u}$ is a ground state solution of \eqref{question} with some $\hat{\lambda}<0$.
\end{theorem}
\begin{proof}
Let us consider a minimizing sequence $\{v_n\}$ for $c_{p,q,a}$. By Lemma \ref{Four}, if $2<q<\bar{p}<p<\frac{2N}{(N-4)^+}$ and $p,q\in 2\mathbb{N}$, then $\{v_n^\sharp\}\in \mathcal{A}_{R_0}$ and
\begin{equation*}
  \mathcal{J}_{p,q}(v_n^\sharp)\leq \mathcal{J}_{p,q}(v_n),
\end{equation*}
which means $\{v_n^\sharp\}$ is a radial minimizing sequence for $c_{p,q,a}$. By Lemma \ref{sx3}, there exists a sequence $\{s_{v_n^\sharp}\}\subset \mathbb{R}$ such that $\{s_{v_n^\sharp}*v_n^\sharp\} \subset \mathcal{P}_{p,q,a}^+\subset \mathcal{A}_{R_0}$ and
\begin{equation*}
  \mathcal{J}_{p,q}(s_{v_n^\sharp}*v_n^\sharp)=\min\Big\{\mathcal{J}_{p,q}(\tau*v_n^\sharp):\tau\in \mathbb{R} \quad\text{and}\quad \|\Delta (\tau*v_n^\sharp)\|_2^2< R_0\Big\}\leq\mathcal{J}_{p,q}(v_n^\sharp),
\end{equation*}
in this way we obtain a new radial minimizing sequence $\{\omega_n:=s_{v_n^\sharp}*v_n^\sharp\}$ for $c_{p,q,a}$ with $\{\omega_n\}\subset S_{a,r}\cap \mathcal{P}_{p,q,a}^+$. By Lemma \ref{sx5}, $\|\Delta \omega_n\|_2^2<R_0-\rho$ for any $n\in \mathbb{N}$. Then by the Ekeland's variational principle, we can find a new minimizing sequence $\{u_n \}$ for $c_{p,q,a}$ with $\|u_n-\omega_n\|\rightarrow0$ as $n\rightarrow\infty$, which is also a Palais-Smale sequence for $\mathcal{J}_{p,q}|_{S_a}$. The condition $\|u_n-\omega_n\|\rightarrow0$ and the boundedness of $\{\omega_n\}$ imply $P_{p,q}(u_n)\rightarrow 0$ as $n\rightarrow\infty$. Hence, $\{u_n\}$ satisfies all the conditions of Lemma \ref{com2}. As a consequence, there exists $\hat{u}\in H^2_{rad}(\mathbb{R}^N)$ such that, up to a subsequence, $u_n\rightarrow \hat{u}$ in $H^2(\mathbb{R}^N)$ as $n\rightarrow\infty$ and $\hat{u}$ is a radial solution of \eqref{question} with some $\hat{\lambda}<0$. 
Moreover, by Lemma \ref{sx5}, we know $\hat{u}$ is a ground state solution.
\end{proof}

Now, we focus on the existence of the second critical point of mountain pass type for $\mathcal{J}_{p,q}|_{S_a}$. To construct a minimax structure, we need the following  lemmas.

\begin{lemma}\label{sx6}
Suppose that $\mathcal{J}_{p,q}(u)<c_{p,q,a}$, then the value $t_u$ defined by Lemma \ref{sx3} is negative.
\end{lemma}
\begin{proof}
Since $t_u<d_u$, if $d_u\leq0$, then $t_u<0$. If $c_u\leq0<d_u$,  then $\mathcal{J}_{p,q}(u)=\Psi_u(0)\geq0$, which is an absurd, since $c_{p,q,a}<0$. Thus $c_u>0$ and
\begin{equation*}
  c_{p,q,a}>\mathcal{J}_{p,q}(u)=\Psi_u(0)\geq \inf\limits_{s<c_u} \Psi_u(s)\geq \min\Big\{\mathcal{J}_{p,q}(\tau*u):\tau\in \mathbb{R}\,\,\text{and}\,\, \|\Delta u\|_2^2<R_0\Big\}=\mathcal{J}_{p,q}(s_u*u)\geq c_{p,q,a},
\end{equation*}
which is again a contradiction.
\end{proof}

\begin{lemma}\label{sx7}
It results that $\tilde{\sigma}_{p,q,a}:=\inf\limits_{u\in \mathcal{P}_{p,q,a}^{-}}\mathcal{J}_{p,q}(u)>0$.
\end{lemma}
\begin{proof}
Let $t_{max}$ denote the strict maximum of the function $h$ at positive level in Lemma \ref{sx1}. For any $u\in \mathcal{P}_{p,q,a}^-$, there exists $\tau_u\in \mathbb{R}$ such that
$\|\Delta (\tau_u*u)\|_2^2=t_{max}$. From Lemma \ref{sx3}, we have
\begin{equation*}
  \mathcal{J}_{p,q}(u)=\Psi_u(0)=\max\limits_{\tau\in \mathbb{R}}\Psi_u(\tau)\geq \Psi_u(\tau_u)=\mathcal{J}_{p,q}(\tau_u*u)\geq h(\|\Delta (\tau_u*u)\|_2^2)=h(t_{max})>0.
\end{equation*}
By the arbitrariness of $u\in \mathcal{P}_{p,q,a}^-$, we obtain $\inf\limits_{u\in \mathcal{P}_{p,q,a}^{-}}\mathcal{J}_{p,q}(u)\geq h(t_{max})>0$.
\end{proof}

Denoting by $\mathcal{J}_{p,q}^{c}:=\Big\{u\in S_{a,r}:\mathcal{J}_{p,q}(u)\leq c\Big\}$, $\mathcal{P}_{p,q,a,r}^{+}:=\mathcal{P}_{p,q,a}^{+}\cap H^2_{rad}(\mathbb{R}^N)$, $\mathcal{P}_{p,q,a,r}^{-}:=\mathcal{P}_{p,q,a}^{-}\cap H^2_{rad}(\mathbb{R}^N)$, $\tilde{\sigma}_{p,q,a,r}:=\inf\limits_{u\in \mathcal{P}_{p,q,a,r}^{-}}\mathcal{J}_{p,q}(u)>0$, $-\infty<c_{p,q,a,r}:=\inf\limits_{u\in \mathcal{A}_{R_0,r}}\mathcal{J}_{p,q}(u)<0$, where
\begin{equation*}
  \mathcal{A}_{R_0,r}:=\Big\{u\in S_{a,r}:\|\Delta u\|_2^2<R_0\}.
\end{equation*}
We introduce the minimax class
\begin{equation*}
  \Gamma:=\Big\{\gamma\in C([0,1],S_{a,r}):\gamma(0)\in \mathcal{P}_{p,q,a,r}^+,\gamma(1)\in \mathcal{J}_{p,q}^{2c_{p,q,a,r}}\Big\},
\end{equation*}
then $\Gamma\neq\emptyset$. Indeed, for any fixed $u\in S_{a,r}$, by Lemmas \ref{biancon} and \ref{sx3}, we have $s_u*u\in \mathcal{P}_{p,q,a,r}^+$, $\mathcal{J}_{p,q}(\tilde{\tau}*u)<2c_{p,q,a,r}$ for $\tilde{\tau}>>1$ large enough, and $\tau\mapsto \tau*u$ is continuous for any $\tau\in \mathbb{R}$, thus for any $t\in[0,1]$, $((1-t)s_u+t\tilde{\tau})*u\in \Gamma$. We define the minimax value
\begin{equation}\label{shanlu1}
  \sigma_{p,q,a,r}:=\inf\limits_{\gamma\in \Gamma}\max\limits_{t\in [0,1]}\mathcal{J}_{p,q}(\gamma(t)).
\end{equation}
\begin{theorem}\label{mp}
$\sigma_{p,q,a,r}>0$ can be achieved by some $\tilde{u}\in S_{a,r}$, and $\tilde{u}$ is a solution of \eqref{question} with some $\tilde{\lambda}<0$.
\end{theorem}
\begin{proof}
Set $\mathcal{F}:=\Gamma$, $A:=\gamma([0,1])$, $F:=\mathcal{P}_{p,q,a,r}^-$ and $B:=\mathcal{P}_{p,q,a,r}^+\cup \mathcal{J}_{p,q}^{2c_{p,q,a,r}}$. By $\mathcal{P}_{p,q,a,r}^0=\emptyset$, we know $B$ is a closed subset of $S_{a,r}$. First, we prove that $\mathcal{F}$ is a homotopy stable family with extended boundary $B$. In fact, for any $\gamma\in \Gamma$ and any $\eta\in C([0,1]\times S_{a,r},S_{a,r})$ satisfying $\eta(t,u)=u$ for $(t,u)\in (\{0\}\times S_{a,r})\cup ([0,1]\times B)$, we have $\eta(\{1\},\gamma(0))=\gamma(0)\in \mathcal{P}_{p,q,a,r}^+$ and $\eta(\{1\},\gamma(1))=\gamma(1)\in \mathcal{J}_{p,q}^{2c_{p,q,a,r}}$. Hence, $\eta(\{1\},\gamma(t))\in \Gamma$, that is $\mathcal{F}$ is a homotopy stable family with extended boundary $B$.

Next we prove
\begin{equation*}
  A\cap F\backslash B\neq\emptyset, \quad\text{for any $A\in \mathcal{F}$}
\end{equation*}
and
\begin{equation*}
  \sup\limits_{u\in B}\mathcal{J}_{p,q}(u)\leq c\leq \inf\limits_{u\in F}\mathcal{J}_{p,q}(u) .
\end{equation*}
As we all know, $A\cap F\backslash B=A\cap (F\backslash B)$, by Lemmas \ref{sx4}, \ref{sx5} and \ref{sx7}, we have $F\cap B=\emptyset$, thus for any $\gamma\in \Gamma$,
\begin{equation}\label{set}
  A\cap F\backslash B=A\cap (F\backslash B)=A\cap F=\gamma([0,1])\cap \mathcal{P}_{p,q,a,r}^-.
\end{equation}
Since $\gamma(0)\in \mathcal{P}_{p,q,a,r}^+$, we have $s_{\gamma(0)}=0$ and hence $t_{\gamma(0)}>s_{\gamma(0)}=0$. On the other hand, by $\mathcal{J}_{p,q}(\gamma(1))\leq 2c_{p,q,a,r}<c_{p,q,a,r}$ and Lemma \ref{sx6}, we obtain $t_{\gamma(1)}<0$. The continuity of $t_{u}$ and $\gamma$ implies that there exists $\tau_\gamma\in (0,1)$ such that $t_{\gamma(\tau_\gamma)}=0$. Using Lemma \ref{sx3}, we have $\gamma(\tau_\gamma)\in \mathcal{P}_{p,q,a,r}^-$. Thus $A\cap F\backslash B\neq\emptyset$.

By \eqref{set}, for any $\gamma\in \Gamma$, we obtain
\begin{equation*}
  \max\limits_{t\in[0,1]}\mathcal{J}_{p,q}(\gamma(t))\geq \inf\limits_{u\in \mathcal{P}_{p,q,a,r}^-}\mathcal{J}_{p,q}(u),
\end{equation*}
thus $\sigma_{p,q,a,r}\geq \tilde{\sigma}_{p,q,a,r}$. On the other hand, for any $u\in \mathcal{P}_{p,q,a,r}^{-}$ and $t\in[0,1]$, $((1-t)s_u+t\tilde{\tau})*u\in \Gamma$. By Lemma \ref{sx3},
\begin{equation*}
  \mathcal{J}_{p,q}(u)=\max\limits_{\tau\in \mathbb{R}}\mathcal{J}_{p,q}(\tau*u)\geq \max\limits_{t\in[0,1]}\mathcal{J}_{p,q}(((1-t)s_u+t\tilde{\tau})*u)\geq \sigma_{p,q,a,r},
\end{equation*}
which implies that $\tilde{\sigma}_{p,q,a,r}\geq \sigma_{p,q,a,r}$. Thus $\sigma_{p,q,a,r}= \tilde{\sigma}_{p,q,a,r}>0$. By Lemmas \ref{sx4}, \ref{sx5} and \ref{sx7}, we obtain
\begin{equation*}
  \sup\limits_{u\in \mathcal{P}_{p,q,a,r}^+\cup \mathcal{J}_{p,q}^{2c_{p,q,a,r}} }\mathcal{J}_{p,q}(u)< \sigma_{p,q,a,r} =\tilde{\sigma}_{p,q,a,r} .
\end{equation*}

From the above arguments, using Lemma \ref{Ghouss}, we can obtain a Palais-Smale sequence $\{u_n\}\subset S_{a,r}$ for $\mathcal{J}_{p,q}|_{S_{a,r}}$ at level $\sigma_{p,q,a,r}>0$ and $\lim\limits_{n\rightarrow\infty}dist (u_n,\mathcal{P}_{p,q,a,r}^-)=0$, i.e., $\lim\limits_{n\rightarrow\infty}P(u_n)=0$. By Lemma \ref{com1}, we deduce that, there exists $\tilde{u}\in H^2_{rad}(\mathbb{R}^N)$ such that, up to a subsequence, $u_n\rightarrow \hat{u}$ in $H^2(\mathbb{R}^N)$ as $n\rightarrow\infty$ and $\tilde{u}$ is a radial solution of \eqref{question} with some $\tilde{\lambda}<0$.
\end{proof}

\textbf{Proof of Theorem \ref{th4}:} Theorem \ref{th4} follows from Theorems \ref{local} and \ref{mp}.
\qed

\subsection{The case $N\geq2$, $2<q\leq \bar{p}<p<\frac{2N}{(N-4)^+}$ and $\mu<0$}

\begin{lemma}\label{xiaoyu1}
$\mathcal{P}_{p,q,a}^0=\emptyset$ and $\mathcal{P}_{p,q,a}$ is a smooth manifold of codimension $2$ in $H^2(\mathbb{R}^N)$.
\end{lemma}
\begin{proof}
If there exists $u\in \mathcal{P}_{p,q,a}^0$, then
\begin{equation*}
 (p\gamma_p-2)\gamma_p\|u\|_p^p= (2-q\gamma_q)\mu\gamma_q\|u\|_q^{q}.
\end{equation*}
By $2<q\leq \bar{p}<p<\frac{2N}{(N-4)^+}$, we have $q\gamma_q\leq2<p\gamma_p$, this together with $\mu<0$ yields that $\|u\|_p=\|u\|_q=0$, which is contradict to $u\in S_a$. The rest of the proof is similar to the one of Lemma \ref{sx2}, we omit it.
\end{proof}

\begin{lemma}\label{xiaoyu2}
For any $u\in S_a$, the function $\Psi_u$ has a unique critical point $t_u\in \mathbb{R}$ and $t_u$ is a strict maximum point at positive level. Moreover:
\\(i)  $t_u*u\in \mathcal{P}_{p,q,a}=\mathcal{P}_{p,q,a}^-$.
\\(ii) $\Psi_u$ is strictly decreasing and concave on $(t_u,\infty)$.
\\(iii)
The map $u\in S_a\mapsto t_u\in \mathbb{R}$ is of class $C^1$.
\\(iv) $P_{p,q}(u)<0$ if and only if $t_u<0$.
\end{lemma}
\begin{proof}
Noticed that, since $\mu<0$, we have $\lim\limits_{\tau\rightarrow-\infty}\Psi_u(\tau)=0^+$ and $\lim\limits_{\tau\rightarrow\infty}\Psi_u(\tau)=-\infty$. Moreover, for any $u\in S_a$, $\Psi_u'(\tau)=0$ is and only if
\begin{equation*}
  \varphi(\tau)=\mu \gamma_q \|u\|^q<0,\quad \text{with}\quad\varphi(\tau):=e^{\frac{N(\bar{p}-q)}{2}\tau}\| \Delta u\|_2^2-\gamma_p e^{\frac{N(p-q)}{2}\tau}\|u\|_p^p.
\end{equation*}
By $0\leq\frac{N(\bar{p}-q)\tau}{2}<\frac{N(p-q)\tau}{2}$, we know the above equation has at most one solution. Thus $\Psi_u$ has a unique critical point $t_u\in \mathbb{R}$ and $t_u$ is a strict maximum point at positive level. If $u\in \mathcal{P}_{p,q,a}$, then $t_u=0$ and $\Psi''_u(0)\leq0$, by $\mathcal{P}_{p,q,a}^0=\emptyset$, we deduce $\Psi''_u(0)<0$, $u\in \mathcal{P}_{p,q,a}^-$, so that $\mathcal{P}_{p,q,a}=\mathcal{P}_{p,q,a}^-$. Since $\Psi_u'(t)<0$ if and only if $t> t_u$, we know $P_{p,q}(u)=\frac{1}{2}\Psi_u'(0)<0$ if and only if $t_u<0$. The rest of the proof is similar to the proof of Lemma \ref{sx3}, so we omit it.
\end{proof}

\begin{lemma}\label{xiaoyu3}
It results that $M_{p,q,a}>0$.
\end{lemma}
\begin{proof}
If $u\in \mathcal{P}_{p,q,a}$, then by $\mu<0$ and Lemmas \ref{GN} and \ref{poho}, we have
\begin{equation*}
  \|\Delta u\|_2^2=\mu\gamma_q\|u\|_q^{q}+\gamma_p\|u\|_p^{p}\leq \gamma_p\|u\|_p^{p}\leq \gamma_p B_{N,p}a^{p-\frac{N(p-2)}{4}}\|\Delta u\|_2^{\frac{N(p-2)}{4}}.
\end{equation*}
Since $p>\bar{p}$, $\frac{N(p-2)}{4}>2$, there exists $\delta>0$ such that $\inf\limits_{u\in \mathcal{P}_{p,q,a}}\|\Delta u\|_2\geq \delta>0$. Using Lemma \ref{poho} again, we obtain
\begin{equation*}
  \mathcal{J}_{p,q}(u)=\Big(\frac{1}{2}-\frac{1}{p\gamma_p}\Big)\int_{\mathbb{R}^N}|\Delta u|^2dx+\Big(\frac{\gamma_q}{p\gamma_p}-\frac{1}{q}\Big)\mu\int_{\mathbb{R}^N}|u|^qdx\geq \Big(\frac{1}{2}-\frac{1}{p\gamma_p}\Big)\delta^2>0.
\end{equation*}
\end{proof}
\begin{lemma}\label{xiaoyu4}
There exists $k_a>0$ small enough such that
\begin{equation*}
  0<\sup\limits_{u\in \overline{\mathcal{A}_{k_a}}}\mathcal{J}_{p,q}(u)<M_{p,q,a},\quad\text{and}\quad u\in \overline{\mathcal{A}_{k_a}}\Rightarrow \mathcal{J}_{p,q}(u),\,P_{p,q}(u)>0.
\end{equation*}
\end{lemma}
\begin{proof}
By lemma \ref{GN}, for any $u\in S_a$, we have
\begin{equation*}
  \mathcal{J}_{p,q}(u)\geq \frac{1}{2}\|\Delta u\|_2^2-\frac{B_{N,p}a^{p-\frac{N(p-2)}{4}}}{p}\|\Delta u\|_2^{\frac{N(p-2)}{4}},
\end{equation*}
\begin{equation*}
  P_{p,q}(u)\geq \|\Delta u\|_2^2-\gamma_pB_{N,p}a^{p-\frac{N(p-2)}{4}}\|\Delta u\|_2^{\frac{N(p-2)}{4}}.
\end{equation*}
Since $\frac{N(p-2)}{4}>2$, choosing $k_a>0$ small enough, and by Lemma \ref{xiaoyu3}, we complete the proof.
\end{proof}

Denoting by $\mathcal{J}_{p,q}^{c}:=\Big\{u\in S_{a,r}:\mathcal{J}_{p,q}(u)\leq c\Big\}$, $\mathcal{P}_{p,q,a,r}:=\mathcal{P}_{p,q,a}\cap  H^2_{rad}(\mathbb{R}^N)$, $M_{p,q,a,r}:=\inf\limits_{u\in \mathcal{P}_{p,q,a,r}}\mathcal{J}_{p,q}(u)>0$, and
\begin{equation*}
  \mathcal{A}_{k_a,r}:=\Big\{u\in S_{a,r}:\|\Delta u\|_2^2<k_a\}.
\end{equation*}
We introduce the minimax class
\begin{equation*}
  \Gamma:=\Big\{\gamma\in C([0,1],S_{a,r}):\gamma(0)\in \overline{\mathcal{A}_{k_a,r}},\gamma(1)\in \mathcal{J}_{p,q}^{0}\Big\},
\end{equation*}
then $\Gamma\neq\emptyset$. Indeed, for any fixed $u\in S_{a,r}$, there exists $\tau_1<<-1$ small enough and $\tau_2>>1$ large enough, such that $\tau_1*u\in \overline{\mathcal{A}_{k_a,r}}$, $\mathcal{J}_{p,q}(\tau_2*u)<0$, and $\tau\mapsto \tau*u$ is continuous for any $\tau\in \mathbb{R}$, thus for any $t\in[0,1]$, $((1-t)\tau_1+t\tau_2)*u\in \Gamma$. We define the minimax value
\begin{equation*}\label{shanlu2}
  \varsigma_{p,q,a,r}:=\inf\limits_{\gamma\in \Gamma}\max\limits_{t\in [0,1]}\mathcal{J}_{p,q}(\gamma(t)).
\end{equation*}
\begin{theorem}\label{xiaoyump}
$\varsigma_{p,q,a,r}>0$ can be achieved by some $\hat{u}\in S_{a,r}$, and $\hat{u}$ is a radial ground state solution of \eqref{question} with some $\hat{\lambda}<0$.
\end{theorem}
\begin{proof}
Set $\mathcal{F}:=\Gamma$, $A:=\gamma([0,1])$, $F:=\mathcal{P}_{p,q,a,r}$ and $B:=\overline{\mathcal{A}_{k_a,r}}\cup \mathcal{J}_{p,q}^{0}$. First, we prove that $\mathcal{F}$ is a homotopy stable family with extended boundary $B$. In fact, for any $\gamma\in \Gamma$ and any $\eta\in C([0,1]\times S_{a,r},S_{a,r})$ satisfying $\eta(t,u)=u$ for $(t,u)\in (\{0\}\times S_{a,r})\cup ([0,1]\times B)$, we have $\eta(\{1\},\gamma(0))=\gamma(0)\in \overline{\mathcal{A}_{k_a,r}}$ and $\eta(\{1\},\gamma(1))=\gamma(1)\in \mathcal{J}_{p,q}^{0}$. Hence, $\eta(\{1\},\gamma(t))\in \Gamma$, that is $\mathcal{F}$ is a homotopy stable family with extended boundary $B$.

Next we prove
\begin{equation*}
  A\cap F\backslash B\neq\emptyset, \quad\text{for any $A\in \mathcal{F}$}
\end{equation*}
and
\begin{equation*}
  \sup\limits_{u\in B}\mathcal{J}_{p,q}(u)\leq c\leq \inf\limits_{u\in F}\mathcal{J}_{p,q}(u) .
\end{equation*}
As we all know, $A\cap F\backslash B=A\cap (F\backslash B)$, by Lemmas \ref{xiaoyu3} and \ref{xiaoyu4}, we have $F\cap B=\emptyset$, thus for any $\gamma\in \Gamma$,
\begin{equation}\label{xiaoyuset}
  A\cap F\backslash B=A\cap (F\backslash B)=A\cap F=\gamma([0,1])\cap \mathcal{P}_{p,q,a,r}.
\end{equation}
Since $\gamma(0)\in \overline{\mathcal{A}_{k_a,r}}$, from Lemma \ref{xiaoyu4}, we have $P_{p,q}({\gamma(0)})>0$.
On the other hand, by $\mathcal{J}_{p,q}(\gamma(1))<0$ and Lemma \ref{xiaoyu2}, we know $t_{\gamma(1)}<0$, thus $P_{p,q}(\gamma(1))<0$. By the continuity of $P_{p,q}(\gamma(t))$, there exists $\tau_\gamma\in (0,1)$ such that $P_{p,q}(\gamma(\tau_\gamma))=0$, namely, $\gamma([0,1])\cap \mathcal{P}_{p,q,a,r}\neq\emptyset$.

By \eqref{xiaoyuset}, for any $\gamma\in \Gamma$, we obtain
\begin{equation*}
  \max\limits_{t\in[0,1]}\mathcal{J}_{p,q}(\gamma(t))\geq \inf\limits_{u\in \mathcal{P}_{p,q,a,r}}\mathcal{J}_{p,q}(u),
\end{equation*}
thus $\varsigma_{p,q,a,r}\geq M_{p,q,a,r}$. On the other hand, for any $u\in \mathcal{P}_{p,q,a,r}$ and any $t\in[0,1]$, $((1-t)\tau_1+t\tau_2)*u\in \Gamma$. By Lemma \ref{xiaoyu2},
\begin{equation*}
  \mathcal{J}_{p,q}(u)=\max\limits_{\tau\in \mathbb{R}}\mathcal{J}_{p,q}(\tau*u)\geq \max\limits_{t\in[0,1]}\mathcal{J}_{p,q}(((1-t)\tau_1+t\tau_2)*u)\geq \varsigma_{p,q,a,r},
\end{equation*}
which implies that $M_{p,q,a,r}\geq \varsigma_{p,q,a,r}$. Thus $\varsigma_{p,q,a,r}= M_{p,q,a,r}>0$. By Lemma \ref{xiaoyu4}, we obtain
\begin{equation*}
\sup\limits_{u\in \overline{\mathcal{A}_{k,r}}\cup \mathcal{J}_{p,q}^{0}}\mathcal{J}_{p,q}(u)<\varsigma_{p,q,a,r}=M_{p,q,a,r}.
\end{equation*}

From the above arguments, using Lemma \ref{Ghouss}, we can obtain a Palais-Smale sequence $\{u_n\}\subset S_{a,r}$ for $\mathcal{J}_{p,q}|_{S_{a,r}}$ at level $\varsigma_{p,q,a,r}>0$ and $\lim\limits_{n\rightarrow\infty}dist (u_n,\mathcal{P}_{p,q,a,r})=0$, i.e., $\lim\limits_{n\rightarrow\infty}P(u_n)=0$. By Lemma \ref{com1}, we deduce that, there exists $\hat{u}\in H^2_{rad}(\mathbb{R}^N)$ such that, up to a subsequence, $u_n\rightarrow \hat{u}$ in $H^2(\mathbb{R}^N)$ as $n\rightarrow\infty$ and $\hat{u}$ is a radial ground state solution of \eqref{question} with some $\hat{\lambda}<0$.
\end{proof}

\textbf{Proof of Theorem \ref{th5}:} Theorem \ref{th5} follows from Theorem \ref{xiaoyump}.
\qed

\subsection{The case $N\geq7$, $2<q<\bar{p}<p=4^*$ and $\mu>0$}\label{crisx}

With similar arguments of Section \ref{mountain}, we have the following facts.

\begin{lemma}\label{crisx1}
$\mathcal{P}_{4^*,q,a}^0=\emptyset$ and $\mathcal{P}_{4^*,q,a}$ is a smooth manifold of codimension $2$ in $H^2(\mathbb{R}^N)$.
\end{lemma}

\begin{lemma}\label{crisx2}
For any $u\in S_a$, the function $\Psi_u$ has exactly two critical points $s_u<t_u\in \mathbb{R}$ and two zeros $c_u<d_u\in \mathbb{R}$, with $s_u<c_u<t_u<d_u$. Moreover, there exist $0<R_0<R_1$ such that
\\(i) $s_u*u\in \mathcal{P}_{4^*,q,a}^+$ and $t_u*u\in \mathcal{P}_{4^*,q,a}^-$. If $\tau*u\in \mathcal{P}_{4^*,q,a}$, then either $\tau=s_u$ or $\tau=t_u$.\\
(ii) For any $\tau\leq c_u$, $\|\Delta (\tau*u)\|_2^2\leq R_0$, and
\begin{equation*}
  \mathcal{J}_{4^*,q}(s_u*u)=\min\Big\{\mathcal{J}_{4^*,q}(\tau*u):\tau\in \mathbb{R} \quad\text{and}\quad \|\Delta (\tau*u)\|_2^2< R_0\Big\}<0.
\end{equation*}
(iii) We have \begin{equation*}
  \mathcal{J}_{4^*,q}(t_u*u)=\max\Big\{\mathcal{J}_{4^*,q}(\tau*u):\tau\in \mathbb{R} \Big\}>0,
\end{equation*}
and $\Psi_u$ is strictly decreasing and concave on $(t_u,\infty)$.\\
(iv)
The maps $u\in S_a\mapsto s_u\in \mathbb{R}$ and $u\in S_a\mapsto t_u\in \mathbb{R}$ are of class $C^1$.
\end{lemma}

\begin{lemma}\label{crisx3}
It results that $M_{4^*,q,a}=\inf\limits_{u\in \mathcal{P}_{4^*,q,a}^{+}}\mathcal{J}_{4^*,q}(u)<0$ and $\tilde{\sigma}_{4^*,q,a}:=\inf\limits_{u\in \mathcal{P}_{4^*,q,a}^{-}}\mathcal{J}_{4^*,q}(u)>0$.
\end{lemma}

\begin{lemma}\label{crisx4}
For $N\geq5$ and $2<p<4^*$, $\limsup\limits_{p \rightarrow4^*}\tilde{\sigma}_{p,q,a}\leq \tilde{\sigma}_{4^*,q,a}$.
\end{lemma}
\begin{proof}
For any $u\in S_a$, by Lemma \ref{sx3}, there exists a unique $t_{u}\in \mathbb{R}$ such that $t_{u}*u\in \mathcal{P}_{p,q,a}^-$. To emphasize the influence of $p,q$, we denote by $t_{p,q,u}$. Then
\begin{equation*}
  e^{4t_{p,q,u}}\|\Delta u\|_2^2=\gamma_p e^{\frac{N(p-2)}{2}t_{p,q,u}}\|u\|_p^p+\mu \gamma_qe^{\frac{N(q-2)}{2}t_{p,q,u}}\|u\|_q^q
\end{equation*}
and
\begin{equation*}
  2e^{4t_{p,q,u}}\|\Delta u\|_2^2<p\gamma_p^2 e^{\frac{N(p-2)}{2}t_{p,q,u}}\|u\|_p^p+\mu q\gamma_q^2e^{\frac{N(q-2)}{2}t_{p,q,u}}\|u\|_q^q.
\end{equation*}
Thus
\begin{equation*}
  \Big(\frac{\mu \gamma_q \|u\|_q^q}{\|\Delta u\|_2^2}\Big)^{\frac{2}{8-N(q-2)}}\leq e^{t_{p,q,u}}\leq \Big(\frac{\|\Delta u\|_2^2}{\gamma_p \|u\|_p^p}\Big)^{\frac{2}{N(p-2)-8}}.
\end{equation*}
By the Lebesgue dominated convergence theorem, we have $\lim\limits_{p\rightarrow 4^*}\|u\|_p=\|u\|_{4^*}$. Thus $e^{t_{p,q,u}}$ is bounded, then ${t_{p,q,u}}$ is bounded. Up to a subsequence, we assume that $t_{p,q,u}\rightarrow t_{q,u}$ as $p \rightarrow4^*$.

From the above inequalities, we have
\begin{equation*}
  e^{4t_{q,u}}\|\Delta u\|_2^2=e^{2\cdot4^*t_{q,u}}\|u\|_{4^*}^{4^*}+\mu \gamma_qe^{\frac{N(q-2)}{2}t_{q,u}}\|u\|_q^q
\end{equation*}
and
\begin{equation*}
  2e^{4t_{q,u}}\|\Delta u\|_2^2\leq 4^*e^{2\cdot4^*t_{q,u}}\|u\|_{4^*}^{4^*}+\mu q\gamma_q^2e^{\frac{N(q-2)}{2}t_{q,u}}\|u\|_q^q,
\end{equation*}
this together with $\mathcal{P}_{4^*,q,a}^0=\emptyset$ yields that $t_{q,u}*u\in \mathcal{P}_{4^*,q,a}^-$. So
\begin{equation*}
  \limsup\limits_{p \rightarrow4^*}\tilde{\sigma}_{p,q,a}\leq \limsup\limits_{p \rightarrow4^*}\mathcal{J}_{p,q}(t_{p,q,u}*u)=\mathcal{J}_{4^*,q}(t_{q,u}*u).
\end{equation*}
By the arbitrariness of $u\in S_a$, we derive the conclusion.
\end{proof}

\begin{lemma}\label{crisx5}
If $N\geq7$, $2<q<\bar{p}<p=4^*$, $\mu>0$. Moreover, if $N=7$, we further assume that $q>\frac{8}{3}$. Then
$\tilde{\sigma}_{4^*,q,a}<M_{4^*,q,a}+\frac{2}{N}S^{\frac{N}{4}}$.
\end{lemma}
\begin{proof}
We adopt some ideas of \cite[Proposition 1.16]{JL0} and \cite[Lemma 3.1]{WW} to complete my proof.
By \cite{MC} and $\mathcal{J}_{4^*,q}$ is even, we know $M_{4^*,q,a}<0$ can achieved by some nonnegative $u^+\in S_a$.  Moreover, by \eqref{nega}, there exists $\lambda^+<0$ such that
\begin{equation}\label{zheng}
  \Delta^2u^+=\lambda^+ u^++\mu(u^+)^{q-1}+(u^+)^{4^*-1},\quad \text {in $\mathbb{R}^N$.}
\end{equation}
For any $\varepsilon>0$, we define
\begin{equation*}
  u_\varepsilon(x)=\varphi(x)U_{\varepsilon,0}(x)=[N(N-4)(N^2-4)]^{\frac{N-4}{8}}\frac{\varphi(x)\varepsilon^{\frac{N-4}{2}}}{(\varepsilon^2+|x|^2)^{\frac{N-4}{2}}},
\end{equation*}
where $\varphi \in C_0^{\infty}(\mathbb{R}^N)$ is a cut-off function such that $0\leq \varphi(x)\leq 1$ for $x\in \mathbb{R}^N$, $\varphi(x)=1$ for $|x|\leq 1$ and $\varphi(x)=0$ for $|x|\geq 2$.
From \cite{BGP}, \cite{CHMS} and \cite{GDW}, when $N\geq5$, we have
\begin{equation*}
  \|\Delta u_\varepsilon\|_2^2=S^{\frac{N}{4}}+O(\varepsilon^{N-4}),
\end{equation*}
\begin{equation*}
  \|u_\varepsilon\|_{4^*}^{4^*}=S^{\frac{N}{4}}+O(\varepsilon^{N}),
\end{equation*}
\begin{align*}
\begin{aligned}
  \begin{split}
  \|u_\varepsilon\|_2^2=\left\{
  \begin{array}{ll}
   C\varepsilon^4+o(\varepsilon^{4}), \quad &\text{if $N>8$},\\
   C\varepsilon^4|\log \varepsilon|+O(\varepsilon^{4}), \quad &\text{if $N=8$},\\
   C\varepsilon^{N-4}+o(\varepsilon^{N-4}), \quad &\text{if $5\leq N\leq 7$},
    \end{array}
    \right.
  \end{split}
  \end{aligned}
  \end{align*}
and
\begin{align*}
\begin{aligned}
  \begin{split}
  \|u_\varepsilon\|_q^q=\left\{
  \begin{array}{ll}
   C\varepsilon^{N-\frac{q(N-4)}{2}}+o(\varepsilon^{N-\frac{q(N-4)}{2}}), \quad &\text{if $q>\frac{N}{N-4}$},\\
   C\varepsilon^{\frac{N}{2}}|\log \varepsilon|+O(\varepsilon^{\frac{N}{2}}), \quad &\text{if $q=\frac{N}{N-4}$},\\
   C\varepsilon^{\frac{q(N-4)}{2}}+o(\varepsilon^{\frac{q(N-4)}{2}}), \quad &\text{if $q<\frac{N}{N-4}$}.
    \end{array}
    \right.
  \end{split}
  \end{aligned}
  \end{align*}

Let $V_{\varepsilon,l}:=u^++lu_\varepsilon$ and
  \begin{equation*}
    W_{\varepsilon,l}:=\Big(\frac{\|V_{\varepsilon,l}\|_2}{a}\Big)^{\frac{N-4}{4}}V_{\varepsilon,l}\Big((\frac{\|V_{\varepsilon,l}\|_2}{a})^{\frac{1}{2}}x\Big).
  \end{equation*}
Then we have
  \begin{equation*}
    \| W_{\varepsilon,l}\|_2=a,\quad \|\Delta W_{\varepsilon,l}\|_2=\|\Delta V_{\varepsilon,l}\|_2,\quad \|W_{\varepsilon,l}\|_{4^*}=\|V_{\varepsilon,l}\|_{4^*},
  \end{equation*}
and
  \begin{equation*}
    \|W_{\varepsilon,l}\|_q=(\frac{\|V_{\varepsilon,l}\|_2}{a})^{\gamma_q-1}\|V_{\varepsilon,l}\|_q.
  \end{equation*}
By Lemma \ref{crisx2}, there exists a unique $t_{\varepsilon,l}\in \mathbb{R}$ such that $t_{\varepsilon,l}*W_{\varepsilon,l}\in \mathcal{P}_{4^*,q,a}^-$, that is
\begin{equation*}
  e^{4t_{\varepsilon,l}}\|\Delta W_{\varepsilon,l} \|_2^2=e^{2\cdot 4^*t_{\varepsilon,l}}\|W_{\varepsilon,l}\|_{4^*}^{4^*}+\mu \gamma_qe^{\frac{N(q-2)}{2}t_{\varepsilon,l}}\|W_{\varepsilon,l}\|_q^q.
\end{equation*}
Thus
\begin{equation*}
  e^{t_{\varepsilon,l}}\leq \Big(\frac{\|\Delta W_{\varepsilon,l} \|_2^2}{\|W_{\varepsilon,l}\|_{4^*}^{4^*}}\Big)^{\frac{1}{2(4^*-2)}}=\Big(\frac{\|\Delta V_{\varepsilon,l} \|_2^2}{\|V_{\varepsilon,l}\|_{4^*}^{4^*}}\Big)^{\frac{1}{2(4^*-2)}},
\end{equation*}
which implies that $t_{\varepsilon,l}\rightarrow -\infty$ as $l\rightarrow\infty$.
Since $u^+\in \mathcal{P}_{4^*,q,a}^+$, by Lemma \ref{crisx2}, we know $t_{\varepsilon,0}>0$ and $t_{\varepsilon,l}$ is continuous for $l$. Therefore, there exists $0<l_\varepsilon<\infty$ such that $t_{\varepsilon,l_\varepsilon}=0$, namely, $W_{\varepsilon,l_\varepsilon}\in \mathcal{P}_{4^*,q,a}^-$. It follows that
\begin{align}\label{ya}
  \tilde{\sigma}_{4^*,q,a}\leq \mathcal{J}_{4^*,q}(W_{\varepsilon,l_\varepsilon})&=\frac{1}{2} \|\Delta W_{\varepsilon,l_\varepsilon}\|_2^2-\frac{\mu}{q} \| W_{\varepsilon,l_\varepsilon}\|_q^q-\frac{1}{4^*} \| W_{\varepsilon,l_\varepsilon}\|_{4^*}^{4^*}\nonumber \\
  &=\frac{1}{2} \|\Delta V_{\varepsilon,l_\varepsilon}\|_2^2-\frac{\mu }{q} (\frac{\|V_{\varepsilon,l_\varepsilon}\|_2}{a})^{q\gamma_q-q}\| V_{\varepsilon,l_\varepsilon}\|_q^q-\frac{1}{4^*} \| V_{\varepsilon,l_\varepsilon}\|_{4^*}^{4^*}.
\end{align}
If $\liminf\limits_{\varepsilon\rightarrow0}l_\varepsilon=0$, then
\begin{equation*}
  \tilde{\sigma}_{4^*,q,a}\leq \liminf\limits_{\varepsilon\rightarrow0}\mathcal{J}_{4^*,q}(W_{\varepsilon,l_\varepsilon})=\mathcal{J}_{4^*,q}(u^+)=M_{4^*,q,a},
\end{equation*}
which is contradict to Lemma \ref{crisx3}. Moreover, $t_{\varepsilon,l}\rightarrow -\infty$ as $l\rightarrow\infty$ implies that $\limsup\limits_{\varepsilon\rightarrow0}l_\varepsilon<\infty$. Thus there exist $t_2>t_1>0$ independent of $\varepsilon$ such that $l_\varepsilon\in [t_1,t_2]$.
Testing \eqref{zheng} with $u_\varepsilon$, we have
\begin{equation*}
  \int_{\mathbb{R}^N}\Delta u^+ \Delta u_\varepsilon dx= \int_{\mathbb{R}^N}\lambda^+ u^+u_\varepsilon dx+ \int_{\mathbb{R}^N}\mu(u^+)^{q-1}u_\varepsilon dx+ \int_{\mathbb{R}^N}(u^+)^{4^*-1}u_\varepsilon dx,
\end{equation*}
which implies that
\begin{align}\label{yaw}
  \|\Delta V_{\varepsilon,l_\varepsilon}\|_2^2&=\|\Delta u^+\|_2^2+l_\varepsilon^2\|\Delta u_\varepsilon\|_2^2+2l_\varepsilon\int_{\mathbb{R}^N}\Delta u^+ \Delta u_\varepsilon dx\nonumber \\
  &=\|\Delta u^+\|_2^2+l_\varepsilon^2\|\Delta u_\varepsilon\|_2^2+2l_\varepsilon\int_{\mathbb{R}^N}\lambda^+ u^+u_\varepsilon dx+ 2l_\varepsilon\int_{\mathbb{R}^N}\mu(u^+)^{q-1}u_\varepsilon dx+ 2l_\varepsilon\int_{\mathbb{R}^N}(u^+)^{4^*-1}u_\varepsilon dx.
\end{align}
For any $a,b\geq0$, we have the following elementary inequality
\begin{equation*}
  (a+b)^q\geq a^q+b^q+qa^{q-1}b, \quad\text {for any $q\geq2$}.
\end{equation*}
Then we can obtain
\begin{equation*}
  \| V_{\varepsilon,l_\varepsilon}\|_q^q\geq \| u^+\|_q^q+l_\varepsilon^q\| u_\varepsilon\|_q^q+ql_\varepsilon\int_{\mathbb{R}^N}(u^+)^{q-1}u_\varepsilon dx,
\end{equation*}
and
\begin{equation}\label{ya1}
  \| V_{\varepsilon,l_\varepsilon}\|_{4^*}^{4^*}\geq \| u^+\|_{4^*}^{4^*}+l_\varepsilon^{4^*}\| u_\varepsilon\|_{4^*}^{4^*}+{4^*}l_\varepsilon\int_{\mathbb{R}^N}(u^+)^{{4^*}-1}u_\varepsilon dx.
\end{equation}

(i) If $N>8$, then
\begin{equation*}
  \| V_{\varepsilon,l_\varepsilon}\|_2^2=a^2+2l_\varepsilon \int_{\mathbb{R}^N}u^+u_\varepsilon dx+C\varepsilon^4.
\end{equation*}
Using the Taylor expansion, we have
\begin{equation*}
  (\frac{\|V_{\varepsilon,l_\varepsilon}\|_2}{a})^{q\gamma_q-q}=1+\frac{(q\gamma_q-q)l_\varepsilon}{a^2}\int_{\mathbb{R}^N}u^+u_\varepsilon dx-C\varepsilon^4.
\end{equation*}
Since $l_\varepsilon\in [t_1,t_2]$, $\frac{N}{N-4}<2<q<\frac{2N}{N-4}$, by the H\"{o}lder and Sobolev inequalities, we obtain
\begin{equation*}
  \frac{(q\gamma_q-q)l_\varepsilon}{a^2}\int_{\mathbb{R}^N}u^+u_\varepsilon dx\leq \frac{(q-q\gamma_q)t_2}{a^2}\|u^+\|_2\|u_\varepsilon\|_2\leq C\varepsilon^4
\end{equation*}
and
\begin{equation*}
  ql_\varepsilon\int_{\mathbb{R}^N}(u^+)^{q-1}u_\varepsilon dx\leq qt_2 \|u^+\|^{q-1}\|u_\varepsilon\|_q\leq C\varepsilon^{\frac{2N-q(N-4)}{2q}}.
\end{equation*}
Hence,
\begin{align}\label{ya3}
  (\frac{\|V_{\varepsilon,l_\varepsilon}\|_2}{a})^{q\gamma_q-q}\| V_{\varepsilon,l_\varepsilon}\|_q^q\geq& \Big(1+\frac{(q\gamma_q-q)l_\varepsilon}{a^2}\int_{\mathbb{R}^N}u^+u_\varepsilon dx-C\varepsilon^4\Big)\Big(\| u^+\|_q^q+l_\varepsilon^q\| u_\varepsilon\|_q^q+ql_\varepsilon\int_{\mathbb{R}^N}(u^+)^{q-1}u_\varepsilon dx\Big)\nonumber \\
  =&\| u^+\|_q^q+l_\varepsilon^q\| u_\varepsilon\|_q^q+ql_\varepsilon\int_{\mathbb{R}^N}(u^+)^{q-1}u_\varepsilon dx+\frac{(q\gamma_q-q)l_\varepsilon}{a^2}\| u^+\|_q^q\int_{\mathbb{R}^N}u^+u_\varepsilon dx-C\| u^+\|_q^q\varepsilon^4\nonumber\\
  =&\| u^+\|_q^q+ql_\varepsilon\int_{\mathbb{R}^N}(u^+)^{q-1}u_\varepsilon dx+\frac{(q\gamma_q-q)l_\varepsilon}{a^2}\| u^+\|_q^q\int_{\mathbb{R}^N}u^+u_\varepsilon dx-C\varepsilon^4+C\varepsilon^{N-\frac{q(N-4)}{2}}.
\end{align}
From \eqref{ya}, \eqref{yaw}, \eqref{ya1} and \eqref{ya3}, using the fact that $\lambda^+ a^2=\mu(\gamma_q-1)\| u^+\|_q^q$, for $\varepsilon>0$ small enough, we can deduce that
\begin{align*}
  \tilde{\sigma}_{4^*,q,a}&\leq \mathcal{J}_{4^*,q}(u^+)+\frac{l_\varepsilon^2}{2}\|\Delta u_\varepsilon\|_2^2-\frac{l_\varepsilon^{4^*}}{4^*}\|u_\varepsilon\|_{4^*}^{4^*}-C\varepsilon^{N-\frac{q(N-4)}{2}}+C\varepsilon^4\\
  &=M_{4^*,q,a}+(\frac{l_\varepsilon^2}{2}-\frac{l_\varepsilon^{4^*}}{4^*})S^{\frac{N}{4}}-C\varepsilon^{N-\frac{q(N-4)}{2}}+C\varepsilon^4\\
  &\leq M_{4^*,q,a}+\frac{2}{N}S^{\frac{N}{4}}-C\varepsilon^{N-\frac{q(N-4)}{2}}+C\varepsilon^4\\
  &<M_{4^*,q,a}+\frac{2}{N}S^{\frac{N}{4}}.
\end{align*}
where we have used the fact that $N-\frac{q(N-4)}{2}<4$.

(ii) If $N=8$, then
\begin{equation*}
  \| V_{\varepsilon,l_\varepsilon}\|_2^2=a^2+2l_\varepsilon \int_{\mathbb{R}^N}u^+u_\varepsilon dx+C\varepsilon^4|\log \varepsilon|.
\end{equation*}
Using the Taylor expansion, we have
\begin{equation*}
  (\frac{\|V_{\varepsilon,l_\varepsilon}\|_2}{a})^{q\gamma_q-q}=1+\frac{(q\gamma_q-q)l_\varepsilon}{a^2}\int_{\mathbb{R}^N}u^+u_\varepsilon dx-C\varepsilon^4|\log \varepsilon|.
\end{equation*}
Since $l_\varepsilon\in [t_1,t_2]$, $\frac{N}{N-4}=2<q<\frac{2N}{N-4}$, by the H\"{o}lder and Sobolev inequalities, we obtain
\begin{equation*}
  \frac{(q\gamma_q-q)l_\varepsilon}{a^2}\int_{\mathbb{R}^N}u^+u_\varepsilon dx\leq \frac{(q-q\gamma_q)t_2}{a^2}\|u^+\|_2\|u_\varepsilon\|_2\leq C\varepsilon^4|\log \varepsilon|
\end{equation*}
and
\begin{equation*}
  ql_\varepsilon\int_{\mathbb{R}^N}(u^+)^{q-1}u_\varepsilon dx\leq qt_2 \|u^+\|_q^{q-1}\|u_\varepsilon\|_q\leq C\varepsilon^{\frac{2N-q(N-4)}{2q}}.
\end{equation*}
Hence,
\begin{align}\label{ya2}
  (\frac{\|V_{\varepsilon,l_\varepsilon}\|_2}{a})^{q\gamma_q-q}\| V_{\varepsilon,l_\varepsilon}\|_q^q\geq&
  \| u^+\|_q^q+ql_\varepsilon\int_{\mathbb{R}^N}(u^+)^{q-1}u_\varepsilon dx+\frac{(q\gamma_q-q)l_\varepsilon}{a^2}\| u^+\|_q^q\int_{\mathbb{R}^N}u^+u_\varepsilon dx-C\varepsilon^4|\log \varepsilon|+C\varepsilon^{8-2q}.
\end{align}
From \eqref{ya}, \eqref{yaw}, \eqref{ya1}, \eqref{ya2}, for $\varepsilon>0$ small enough, we can deduce that
\begin{align*}
  \tilde{\sigma}_{4^*,q,a}&\leq \mathcal{J}_{4^*,q}(u^+)+\frac{l_\varepsilon^2}{2}\|\Delta u_\varepsilon\|_2^2-\frac{l_\varepsilon^{4^*}}{4^*}\|u_\varepsilon\|_{4^*}^{4^*}-C\varepsilon^{8-2q}+C\varepsilon^4|\log \varepsilon|
  <M_{4^*,q,a}+\frac{2}{N}S^{\frac{N}{4}},
\end{align*}
where we have used the fact that $\lim\limits_{\varepsilon\rightarrow0}\frac{\varepsilon^{4-2q}}{|\log \varepsilon|}=\infty$.

(iii) If $N=7$ and $\frac{8}{3}<q<\frac{22}{7}$, then
\begin{equation*}
  \| V_{\varepsilon,l_\varepsilon}\|_2^2=a^2+2l_\varepsilon \int_{\mathbb{R}^N}u^+u_\varepsilon dx+C\varepsilon^{3}.
\end{equation*}
Using the Taylor expansion, we have
\begin{equation*}
  (\frac{\|V_{\varepsilon,l_\varepsilon}\|_2}{a})^{q\gamma_q-q}=1+\frac{(q\gamma_q-q)l_\varepsilon}{a^2}\int_{\mathbb{R}^N}u^+u_\varepsilon dx-C\varepsilon^{3}.
\end{equation*}
Since $l_\varepsilon\in [t_1,t_2]$, $\frac{N}{N-4}=\frac{7}{3}<q<\frac{2N}{N-4}$, by the H\"{o}lder and Sobolev inequalities, we obtain
\begin{equation*}
  \frac{(q\gamma_q-q)l_\varepsilon}{a^2}\int_{\mathbb{R}^N}u^+u_\varepsilon dx\leq \frac{(q-q\gamma_q)t_2}{a^2}\|u^+\|_2\|u_\varepsilon\|_2\leq C\varepsilon^3
\end{equation*}
and
\begin{equation*}
  ql_\varepsilon\int_{\mathbb{R}^N}(u^+)^{q-1}u_\varepsilon dx\leq qt_2 \|u^+\|_q^{q-1}\|u_\varepsilon\|_q\leq C\varepsilon^{\frac{2N-q(N-4)}{2q}}.
\end{equation*}
Hence,
\begin{align}\label{ya4}
  (\frac{\|V_{\varepsilon,l_\varepsilon}\|_2}{a})^{q\gamma_q-q}\| V_{\varepsilon,l_\varepsilon}\|_q^q\geq&\| u^+\|_q^q+ql_\varepsilon\int_{\mathbb{R}^N}(u^+)^{q-1}u_\varepsilon dx+\frac{(q\gamma_q-q)l_\varepsilon}{a^2}\| u^+\|_q^q\int_{\mathbb{R}^N}u^+u_\varepsilon dx-C\varepsilon^{3}+C\varepsilon^{7-\frac{3q}{2}}.
\end{align}
From \eqref{ya}, \eqref{yaw}, \eqref{ya1}, \eqref{ya4}, for $\varepsilon>0$ small enough, we can deduce that
\begin{align*}
  \tilde{\sigma}_{4^*,q,a}&\leq \mathcal{J}_{4^*,q}(u^+)+\frac{l_\varepsilon^2}{2}\|\Delta u_\varepsilon\|_2^2-\frac{l_\varepsilon^{4^*}}{4^*}\|u_\varepsilon\|_{4^*}^{4^*}-C\varepsilon^{7-\frac{3q}{2}}+C\varepsilon^{3}
  <M_{4^*,q,a}+\frac{2}{N}S^{\frac{N}{4}},
\end{align*}
since $7-\frac{3q}{2}<3$.
\end{proof}

For any $u\in \mathcal{P}_{4^*,q,a}^+$, then $u_b=\frac{b}{a}u\in S_b$ for any $b>0$. By Lemma \ref{crisx2}, there exists a unique $t_{u_b}\in \mathbb{R}$ such that $t_{u_b} *u_b\in \mathcal{P}_{4^*,q,b}^+$. For convenience, we denote by $t_b$. Clearly, $t_a=0$.
\begin{lemma}\label{crisx6}
For $a$ satisfies \eqref{con3} and  \eqref{con1}, $\partial _at_a$ exists and
\begin{equation*}
  \partial _at_a=\frac{4^*\|u\|_{4^*}^{4^*}+\mu q\gamma_q\|u\|_q^q -2\|\Delta u\|_2^2}{2a(2\|\Delta u\|_2^2-4^*\|u\|_{4^*}^{4^*}-\mu q\gamma_q^2\|u\|_q^q)}.
\end{equation*}
\end{lemma}
\begin{proof}
Since $t_b *u_b\in \mathcal{P}_{4^*,q,b}^+$, we have
\begin{equation*}
  (\frac{b}{a})^2e^{4t_{b}}\|\Delta u\|_2^2=(\frac{b}{a})^{4^*}e^{2\cdot4^*t_{b}}\|u\|_{4^*}^{4^*}+\mu \gamma_q (\frac{b}{a})^qe^{\frac{N(q-2)}{2}t_{b}}\|u\|_q^q.
\end{equation*}
Define
\begin{equation*}
  \Phi(b,t)=(\frac{b}{a})^2e^{4t_{}}\|\Delta u\|_2^2-(\frac{b}{a})^{4^*}e^{2\cdot4^*t_{}}\|u\|_{4^*}^{4^*}-\mu \gamma_q (\frac{b}{a})^qe^{\frac{N(q-2)}{2}t_{}}\|u\|_q^q.
\end{equation*}
Then by $u\in \mathcal{P}_{4^*,q,a}^+$ and $\mathcal{P}_{4^*,q,a}^0=\emptyset$, we have $\Phi(a,t_a)=\Phi(a,0)=0$ and
\begin{equation*}
\partial _t\Phi(a,t_a)=\partial _t\Phi(a,0)=4\|\Delta u\|_2^2-2\cdot 4^*\|u\|_{4^*}^{4^*}-2\mu q\gamma_q^2\|u\|_q^q=\frac{1}{2}\Psi''_u(0)>0.
\end{equation*}
By using the implicit function theorem, we know $\partial _at_a$ exists and
\begin{equation*}
 \partial _at_a=-\frac{\partial _b\Phi(b,t)}{\partial _{t}\Phi(b,t)}\Big|_{b=a,t=t_a}=\frac{4^*\|u\|_{4^*}^{4^*}+\mu q\gamma_q\|u\|_q^q -2\|\Delta u\|_2^2}{2a(2\|\Delta u\|_2^2-4^*\|u\|_{4^*}^{4^*}-\mu q\gamma_q^2\|u\|_q^q)}.
\end{equation*}
\end{proof}

\begin{lemma}\label{crisx7}
For $a$ satisfies \eqref{con3} and \eqref{con1}, $M_{4^*,q,a}$ is strictly decreasing in $a$.
\end{lemma}
\begin{proof}
Since
\begin{equation*}
  \mathcal{J}_{4^*,q}(t_b *u_b)=\frac{1}{2} (\frac{b}{a})^2e^{4t_{b}}\|\Delta u\|_2^2-\frac{1}{4^*}(\frac{b}{a})^{4^*}e^{2\cdot4^*t_{b}}\|u\|_{4^*}^{4^*}-\frac{\mu }{q} (\frac{b}{a})^qe^{\frac{N(q-2)}{2}t_{b}}\|u\|_q^q.
\end{equation*}
 For any $b>a$ close to $a$, by Lemma \ref{crisx6}, $u\in \mathcal{P}_{4^*,q,a}^+$, and $\gamma_q<1$, we have
\begin{equation*}
  \partial_b \mathcal{J}_{4^*,q}(t_b *u_b)|_{b=a}=\frac{\|\Delta u\|_2^2-\|u\|_{4^*}^{4^*}-\mu\|u\|_q^q}{a}+2(\|\Delta u\|_2^2-\|u\|_{4^*}^{4^*}-\mu\gamma_q\|u\|_q^q)\partial _at_a=\frac{\mu(\gamma_q-1)}{a}\|u\|_q^q<0,
\end{equation*}
which implies that $\mathcal{J}_{4^*,q}(t_b *u_b)<\mathcal{J}_{4^*,q}(u)$. Since $u^+\in \mathcal{P}_{4^*,q,a}^+$ and $\mathcal{J}_{4^*,q}(u^+)=M_{4^*,q,a}$, we have
\begin{equation*}
  M_{4^*,q,b}\leq \mathcal{J}_{4^*,q}(t_b *u_b)<\mathcal{J}_{4^*,q}(u^+)=M_{4^*,q,a}.
\end{equation*}
\end{proof}

Denoting by $\mathcal{J}_{4^*,q}^{c}:=\Big\{u\in S_{a,r}:\mathcal{J}_{4^*,q}(u)\leq c\Big\}$, $\mathcal{P}_{4^*,q,a,r}:=\mathcal{P}_{4^*,q,a}\cap H^2_{rad}(\mathbb{R}^N)$, $\mathcal{P}_{4^*,q,a,r}^{+}:=\mathcal{P}_{4^*,q,a}^{+}\cap H^2_{rad}(\mathbb{R}^N)$, $\mathcal{P}_{4^*,q,a,r}^{-}:=\mathcal{P}_{4^*,q,a}^{-}\cap H^2_{rad}(\mathbb{R}^N)$, $\tilde{\sigma}_{4^*,q,a,r}:=\inf\limits_{u\in \mathcal{P}_{4^*,q,a,r}^{-}}\mathcal{J}_{4^*,q}(u)$, $-\infty<c_{4^*,q,a,r}:=\inf\limits_{u\in \mathcal{A}_{R_0,r}}\mathcal{J}_{4^*,q}(u)<0$, where
\begin{equation*}
  \mathcal{A}_{R_0,r}:=\Big\{u\in S_{a,r}:\|\Delta u\|_2^2<R_0\}.
\end{equation*}
Let
\begin{equation}\label{M4^*}
M_{4^*,q,a,r}:=\inf\limits_{u\in \mathcal{P}_{4^*,q,a,r}}\mathcal{J}_{4^*,q}(u).
\end{equation}
We introduce the minimax class
\begin{equation*}
  \Gamma:=\Big\{\gamma\in C([0,1],S_{a,r}):\gamma(0)\in \mathcal{P}_{4^*,q,a,r}^+,\gamma(1)\in \mathcal{J}_{4^*,q}^{2c_{4^*,q,a,r}}\Big\},
\end{equation*}
and define the minimax value
\begin{equation}\label{shanlu2}
  \sigma_{4^*,q,a,r}:=\inf\limits_{\gamma\in \Gamma}\max\limits_{t\in [0,1]}\mathcal{J}_{4^*,q}(\gamma(t)).
\end{equation}
Then we have $\sigma_{4^*,q,a,r}=\tilde{\sigma}_{4^*,q,a,r}$.

\textbf{Proof of Theorem \ref{th8}:}
We use the Sobolev subcritical approximation method to complete my proof.
Let $q<\bar{p}<p_n<4^*$, and $\lim\limits_{n\rightarrow\infty}p_n=4^*$. By Theorem \ref{th4} and Lemma \ref{crisx4}, there exists a sequence $\{u_{p_n}^-\in \mathcal{\mathcal{P}}_{p_n,q,a,r}^-\}$ such that $\mathcal{J}_{p_n,q}(u_{p_n}^-)=\sigma_{p_n,q,a,r}=\tilde{\sigma}_{p_n,q,a,r}<\tilde{\sigma}_{4^*,q,a,r}+1$. Moreover, there exists a negative sequence $\{\lambda_{p_n}^-\}$ such that
\begin{equation}\label{fu}
  \Delta^2u_{p_n}^-=\lambda_{p_n}^- u_{p_n}^-+\mu|u_{p_n}^-|^{q-2}u_{p_n}^-+|u_{p_n}^-|^{p_n-2}u_{p_n}^-,\quad \text {in $\mathbb{R}^N$}.
\end{equation}
We divide the proof into six steps.

\textbf{Setp 1:} $\{u_{p_n}^-\}$ is bounded in $H^2(\mathbb{R}^N)$.

By the first equation of \eqref{bdd}, we know $\{u_{p_n}^-\}$ is bounded in $H^2(\mathbb{R}^N)$. Thus there exists $u\in H^2_{rad}(\mathbb{R}^N)$ such that, up to a subsequence, $u_{p_n}^-\rightharpoonup u$ in $H^2(\mathbb{R}^N)$, $u_{p_n}^-\rightarrow u$ in $L^p(\mathbb{R}^N)$ for $2<p<\frac{2N}{(N-4)^+}$ and a.e. in $\mathbb{R}^N$.
From  \eqref{pohoim} and \eqref{fu}, we have
\begin{equation}\label{nega1}
  \lambda_{p_n}^- a^2=\frac{\mu[q(N-4)-2N]}{4q}\|u_{p_n}^-\|_q^q+\frac{p_n(N-4)-2N}{4p_n}\|u_{p_n}^-\|_{p_n}^{p_n}
\end{equation}
Thus $\{\lambda_{p_n}^- \}$ is bounded in $\mathbb{R}$, and there exists $\lambda\in \mathbb{R}$ such that, up to a subsequence, $\lambda_{p_n}^- \rightarrow \lambda\leq0$ as $n\rightarrow\infty$.\\
\textbf{Setp 2:} $u$ is a weak solution of the equation
\begin{equation}\label{jie}
  \Delta^2u=\lambda u+\mu|u|^{q-2}u+|u|^{4^*-2}u,\quad \text {in $\mathbb{R}^N$}.
\end{equation}

 Since $|u_{p_n}^-|^{p_n-2}u_{p_n}^-\rightarrow |u|^{4^*-2}u$ a.e. in $\mathbb{R}^N$, and $|u_{p_n}^-|^{p_n-2}u_{p_n}^-$ is bounded in $L^{\frac{4^*}{4^*-1}}(\mathbb{R}^N)$, by Lemma \ref{weakcon}, we know $|u_{p_n}^-|^{p_n-2}u_{p_n}^-\rightharpoonup |u|^{4^*-2}u$ in $L^{\frac{4^*}{4^*-1}}(\mathbb{R}^N)$. This together with the density of $ C_0^\infty(\mathbb{R}^N)$ in $L^{4^*}(\mathbb{R}^N)$ implies that, for any $\psi\in C_0^\infty(\mathbb{R}^N)$,
\begin{equation*}
  \int_{\mathbb{R}^N}|u_{p_n}^-|^{p_n-2}u_{p_n}^-\psi dx\rightarrow \int_{\mathbb{R}^N}|u|^{4^*-2}u\psi dx,\quad \text{as $n\rightarrow\infty$}.
\end{equation*}
It follows that
\begin{align*}
  0&=\int_{\mathbb{R}^N}\Delta u_{p_n}^- \Delta \psi dx-\int_{\mathbb{R}^N}\lambda_{p_n}^- u_{p_n}^-\psi dx-\int_{\mathbb{R}^N}\mu|u_{p_n}^-|^{q-2}u_{p_n}^-\psi dx-\int_{\mathbb{R}^N}|u_{p_n}^-|^{p_n-2}u_{p_n}^-\psi dx\\
  &\rightarrow \int_{\mathbb{R}^N}\Delta u \Delta \psi dx-\int_{\mathbb{R}^N}\lambda u\psi dx-\int_{\mathbb{R}^N}\mu|u|^{q-2}u\psi dx-\int_{\mathbb{R}^N}|u|^{4^*-2}u\psi dx,\quad \text{as $n\rightarrow\infty$},
\end{align*}
which indicates that $u$ is a weak solution of \eqref{jie}.\\
\textbf{Setp 3:} $u_{p_n}^-\rightarrow u$ in $D^{2,2}(\mathbb{R}^N)$ as $n\rightarrow\infty$.

Let $v_{p_n}^-=u_{p_n}^--u\rightharpoonup0$ in $H^2(\mathbb{R}^N)$. By the Br\'{e}zis-Lieb lemma \cite{Wil}, we have
\begin{equation}\label{BL1}
  \|u_{p_n}^-\|_q^q=\|u\|_q^q+o_n(1),\quad \|u_{p_n}^-\|_{4^*}^{4^*}=\|v_{p_n}^-\|_{4^*}^{4^*}+\|u\|_{4^*}^{4^*}+o_n(1).
\end{equation}
Define
\begin{equation*}
  f_u(v):=\int_{\mathbb{R}^N} \Delta u  \Delta vdx, \quad \text{for any $v\in H^2(\mathbb{R}^N)$}.
\end{equation*}
Then $|f_u(v)|\leq \|u\|\|v\|$, which means that $f_u$ is a linear and continuous functional on $H^2(\mathbb{R}^N)$. So by the definition of  weak convergence, we have
\begin{equation*}
  \int_{\mathbb{R}^N}\Delta v_{p_n}^- \Delta u dx\rightarrow0,\quad \text{as $n\rightarrow\infty$}.
\end{equation*}
Thus
\begin{equation}\label{BL2}
  \|\Delta u_{p_n}^-\|_2^2= \|\Delta v_{p_n}^-\|_2^2+ \|\Delta u\|_2^2+o_n(1).
\end{equation}
Since $u_{p_n}^-\in \mathcal{\mathcal{P}}_{p_n,q,a,r}^-$ and $\lim\limits_{p\rightarrow 4^*}\|u\|_p=\|u\|_{4^*}$, we have
\begin{align}\label{BL3}
  \|\Delta u_{p_n}^-\|_2^2=\gamma_{p_n}\|u_{p_n}^-\|_{p_n}^{p_n}+\mu\gamma_q\|u_{p_n}^-\|_q^q=\|u_{p_n}^-\|_{4^*}^{4^*}+\mu\gamma_q\|u_{p_n}^-\|_q^q+o_n(1).
\end{align}
From \eqref{jie}-\eqref{BL3} and the definition of $S$, we deduce that
\begin{equation*}
  \|\Delta v_{p_n}^-\|_2^2=\|v_{p_n}^-\|_{4^*}^{4^*}+o_n(1)\leq S^{-\frac{4^*}{2}}\|\Delta v_{p_n}^-\|_2^{4^*}+o_n(1).
\end{equation*}
Assume that $\|\Delta v_{p_n}^-\|_2^2\rightarrow m$ as $n\rightarrow \infty$. Then from the above inequality, we know $m=0$ or $m\geq S^{\frac{N}{4}}$. If $m\geq S^{\frac{N}{4}}$, by \eqref{BL1}, \eqref{BL2}, Lemmas \ref{crisx3}, \ref{crisx4} and \ref{crisx7}, we obtain
\begin{align*}
  \tilde{\sigma}_{4^*,q,a,r}&\geq \limsup\limits_{n\rightarrow\infty}\tilde{\sigma}_{p_n,q,a,r}\\
  &=\limsup\limits_{n\rightarrow\infty}\Big[\Big(\frac{1}{2}-\frac{1}{p_n\gamma_{p_n}}\Big)\int_{\mathbb{R}^N}|\Delta u_{p_n}^-|^2dx-\mu \gamma_q\Big(\frac{1}{q\gamma_q}-\frac{1}{p_n\gamma_{p_n}}\Big)\int_{\mathbb{R}^N}|u_{p_n}^-|^qdx\Big]\\
  &=\limsup\limits_{n\rightarrow\infty}\Big[\Big(\frac{1}{2}-\frac{1}{4^*}\Big)\int_{\mathbb{R}^N}|\Delta u_{p_n}^-|^2dx-\mu \gamma_q\Big(\frac{1}{q\gamma_q}-\frac{1}{4^*}\Big)\int_{\mathbb{R}^N}|u_{p_n}^-|^qdx\Big]\\
  &=\limsup\limits_{n\rightarrow\infty}\Big(\frac{1}{2}-\frac{1}{4^*}\Big)\int_{\mathbb{R}^N}|\Delta v_{p_n}^-|^2dx+\Big[\Big(\frac{1}{2}-\frac{1}{4^*}\Big)\int_{\mathbb{R}^N}|\Delta u|^2dx
  -\mu \gamma_q\Big(\frac{1}{q\gamma_q}-\frac{1}{4^*}\Big)\int_{\mathbb{R}^N}|u|^qdx\Big]\\
  &=\limsup\limits_{n\rightarrow\infty}\Big(\frac{1}{2}-\frac{1}{4^*}\Big)\int_{\mathbb{R}^N}|\Delta v_{p_n}^-|^2dx+\mathcal{J}_{4^*,q}(u)\\
  &\geq \frac{2}{N}S^{\frac{N}{4}}+\min \Big\{0,M_{4^*,q,\|u\|,r}\Big\}\\&=\frac{2}{N}S^{\frac{N}{4}}+M_{4^*,q,\|u\|,r}\\&\geq \frac{2}{N}S^{\frac{N}{4}}+M_{4^*,q,a,r},
\end{align*}
which is contradict to Lemma \ref{crisx5}. Hence, $m=0$, namely, $u_{p_n}^-\rightarrow u$ in $D^{2,2}(\mathbb{R}^N)$ as $n\rightarrow\infty$.\\
\textbf{Setp 4:} $u\neq0$.

Since $u_{p_n}^-\in \mathcal{\mathcal{P}}_{p_n,q,a,r}^-$,
we have
\begin{align*}
  \|\Delta u_{p_n}^-\|_2^2=\gamma_{p_n}\|u_{p_n}^-\|_{p_n}^{p_n}+\mu\gamma_q\|u_{p_n}^-\|_q^q,
\end{align*}
and
\begin{align}\label{tomp}
  2\|\Delta u_{p_n}^-\|_2^2<p_n\gamma_{p_n}^2\|u_{p_n}^-\|_{p_n}^{p_n}+\mu q\gamma_q^2\|u_{p_n}^-\|_q^q.
\end{align}
From the above inequalities, by using Lemmas \ref{GN} and \ref{limits}, we have
\begin{equation*}
  (2-q\gamma_q)\|\Delta u_{p_n}^-\|_2^2<\gamma_{p_n}(p_n\gamma_{p_n}-q\gamma_q)\|u_{p_n}^-\|_{p_n}^{p_n}\leq \gamma_{p_n}(p_n\gamma_{p_n}-q\gamma_q)B_{N,p_n}a^{p_n(1-\gamma_{p_n})}\|\Delta u_{p_n}^-\|_2^{p_n\gamma_{p_n}},
\end{equation*}
thus
\begin{equation*}
 0< 2-q\gamma_q\leq (4^*-q\gamma_q)S^{-\frac{4^*}{2}}\|\Delta u\|_2^{4^*-2},
\end{equation*}
which implies $u\neq0$. Using \eqref{nega} and \eqref{jie}, we know $\lambda<0$.\\
\textbf{Setp 5:} $u_{p_n}^-\rightarrow u$ in $H^{2}(\mathbb{R}^N)$ as $n\rightarrow\infty$.

Multiply $v_{p_n}^-$ on both sides of \eqref{fu} and \eqref{jie} and subtract, we obtain
\begin{equation*}
  \int_{\mathbb{R}^N}|\Delta v_{p_n}^-|^2dx-\lambda\int_{\mathbb{R}^N}|v_{p_n}^-|^2dx=\int_{\mathbb{R}^N}(|u_{p_n}^-|^{p_n-2}u_{p_n}^--|u|^{4^*-2}u)v_{p_n}^-dx+\mu\int_{\mathbb{R}^N}(|u_{p_n}^-|^{q-2}u_{p_n}^--|u|^{q-2}u)v_{p_n}^-dx.
\end{equation*}
It's obvious that $|u_{p_n}^-|^{p_n-2}u_{p_n}^--|u|^{4^*-2}u \in L^{\frac{4^*}{4^*-1}}(\mathbb{R}^N)$, this together with $v_{p_n}^-\rightarrow0$ in $L^{4^*}(\mathbb{R}^N)$ yields that
\begin{equation*}
  \int_{\mathbb{R}^N}(|u_{p_n}^-|^{p_n-2}u_{p_n}^--|u|^{4^*-2}u)v_{p_n}^-dx\rightarrow0, \quad \text{as $n\rightarrow\infty$}.
\end{equation*}
Moreover, by $|u_{p_n}^-|^{q-2}u_{p_n}^--|u|^{q-2}u\rightarrow0$ a.e. in $\mathbb{R}^N$, and $|u_{p_n}^-|^{q-2}u_{p_n}^--|u|^{q-2}u$ is bounded in $L^{\frac{q}{q-1}}(\mathbb{R}^N)$, using Lemma \ref{weakcon}, we have $|u_{p_n}^-|^{q-2}u_{p_n}^--|u|^{q-2}u\rightharpoonup 0$ in $L^{\frac{q}{q-1}}(\mathbb{R}^N)$. This with $v_{p_n}^-
\in L^q(\mathbb{R}^N)$ yields that
\begin{equation*}
  \int_{\mathbb{R}^N}(|u_{p_n}^-|^{q-2}u_{p_n}^--|u|^{q-2}u)v_{p_n}^-dx\rightarrow0, \quad \text{as $n\rightarrow\infty$}.
\end{equation*}
From the above arguments, by $\lambda<0$, we immediately know $\|v_{p_n}^-\|_2\rightarrow0$, which means $u_{p_n}^-\rightarrow u$ in $L^2(\mathbb{R}^N)$ as $n\rightarrow\infty$. So $u_{p_n}^-\rightarrow u$ in $H^2(\mathbb{R}^N)$ as $n\rightarrow\infty$.\\
\textbf{Setp 6:} $\tilde{\sigma}_{4^*,q,a,r}$ can be achieved.

In fact, by \eqref{tomp} and $\mathcal{P}_{4^*,q,a,r}=\emptyset$, we have
\begin{align*}
  2\|\Delta u\|_2^2<4^*\|u\|_{4^*}^{4^*}+\mu q\gamma_q^2\|u\|_q^q,
\end{align*}
which means $u\in \mathcal{P}_{4^*,q,a,r}^-$. Thus
\begin{equation*}
  \tilde{\sigma}_{4^*,q,a,r}\leq \mathcal{J}_{4^*,q}(u)=\lim\limits_{n\rightarrow\infty} \mathcal{J}_{p_n,q}(u_{p_n}^-)=\lim\limits_{n\rightarrow\infty}\tilde{\sigma}_{p_n,q,a,r}\leq \tilde{\sigma}_{4^*,q,a,r}.
\end{equation*}
\qed

\end{document}